  \documentclass[11pt,letterpaper,twoside]{article}

  \usepackage[margin=3cm]{geometry}
  \usepackage[margin=1cm]{caption}

  \usepackage{amsmath,amssymb,amsthm,amsfonts,bbding,stmaryrd,dsfont,wasysym,pifont,xspace,xcolor}
  \usepackage{parskip,fancyhdr,url}

  \usepackage{fourier}

  \usepackage{epstopdf,yfonts}
  \usepackage{graphicx,enumerate,epsfig} 
  \usepackage{mathrsfs}
  \usepackage{enumitem}
  \usepackage[normalem]{ulem}

  \usepackage[
     breaklinks,colorlinks,
     citecolor=blue,linkcolor=blue,urlcolor=teal
  ]{hyperref}
  
  \usepackage{tikz}
  \usetikzlibrary{arrows,calc,decorations.pathmorphing,backgrounds,positioning,fit}

  \usepackage[all]{xy}

  \makeatletter
    \def\tagform@#1{\maketag@@@{%
     \textbf{(\ignorespaces#1\unskip\@@italiccorr)}}}%
     \renewcommand{\eqref}[1]{\textup{\maketag@@@{(\ignorespaces%
          {\ref{#1}}\unskip\@@italiccorr)}}}
  \makeatother

  \begingroup
    \makeatletter
    \@for\theoremstyle:=definition,remark,plain\do{%
      \expandafter\g@addto@macro\csname th@\theoremstyle\endcsname{%
        \addtolength\thm@preskip\parskip
        }%
      }
  \endgroup

  \theoremstyle{plain}
  \newtheorem{theorem}{Theorem}[section]
  \newtheorem{proposition}[theorem]{Proposition}
  \newtheorem{corollary}[theorem]{Corollary}
  \newtheorem{lemma}[theorem]{Lemma}

  \newtheorem{introthm}{Theorem}
  \newtheorem{introcor}[introthm]{Corollary}

  \theoremstyle{definition}
  \newtheorem{definition}[theorem]{Definition}
  \newtheorem{remark}[theorem]{Remark}
  \newtheorem{remarks}[theorem]{Remarks}
  
  \newtheorem*{claim*}{Claim}

  \newtheorem*{question*}{Question}
  \newtheorem*{answer*}{Answer}
  \newtheorem*{application*}{Application}
  \newtheorem*{notation}{Notation}

  \newcommand{\secref}[1]{Section~\ref{Sec:#1}}
  \newcommand{\thmref}[1]{Theorem~\ref{Thm:#1}}
  \newcommand{\corref}[1]{Corollary~\ref{Cor:#1}}
  \newcommand{\lemref}[1]{Lemma~\ref{Lem:#1}}
  \newcommand{\propref}[1]{Proposition~\ref{Prop:#1}}
  
  \newcommand{\remref}[1]{Remark~\ref{Rem:#1}}

  \newcommand{\figref}[1]{Figure~\ref{Fig:#1}}
  \newcommand{\defref}[1]{Definition~\ref{Def:#1}}

  \newcommand{\Z}{\ensuremath{\mathbb{Z}}\xspace}
  \newcommand{\R}{\ensuremath{\mathbb{R}}\xspace}

  \newcommand{\QH}{\ensuremath{\operatorname{QH}}\xspace} 

  \DeclareMathOperator{\scl}{scl}

  \newcommand{\set}[1]{\ensuremath{\left\{ {#1} \right\}}\xspace} 
  \newcommand{\abs}[1]{\ensuremath{\left\lvert {#1} \right\rvert}\xspace} 
  \newcommand{\gen}[1]{\ensuremath{\left\langle {#1}
      \right\rangle}\xspace} 

  \newcommand{\st}{\ensuremath{\,\, \colon \,\,}\xspace} 
  \newcommand{\from}{\ensuremath{\colon \thinspace}\xspace} 
  \newcommand{\into}{\ensuremath{\hookrightarrow}\xspace} 

  \renewcommand{\(}{\ensuremath{\big(}\xspace}
  \renewcommand{\)}{\ensuremath{\big)}\xspace}

  \newcommand{\g}{\ensuremath{\gamma}\xspace} 
  \newcommand{\bg}{\ensuremath{\overline{\gamma}}\xspace} 

  \renewcommand{\O}{\ensuremath{x_0}\xspace} 

  \newcommand{\calH}{\ensuremath{\mathcal{H}}\xspace}
  
  \newcommand{\half}{\ensuremath{\calH}\xspace} %

  \newcommand{\bH}{\ensuremath{\overline{H}}\xspace} 
  \newcommand{\bK}{\ensuremath{\overline{K}}\xspace}
  \newcommand{\balpha}{\ensuremath{\overline{\alpha}}\xspace}
  
  \newcommand{\bgamma}{\ensuremath{\overline{\gamma}}\xspace}
  \newcommand{\bsigma}{\ensuremath{\overline{\sigma}}\xspace}

  \newcommand{\trans}{\ensuremath{\pitchfork}\xspace} 
  
  \newcommand{\axis}{\ensuremath{A}\xspace} 
  \newcommand{\ess}{\ensuremath{X^{\text{ess}}}\xspace} 
  \newcommand{\mess}{\ensuremath{M^{\text{ess}}}\xspace} 
  \newcommand{\fix}{\ensuremath{X^{\text{fix}}}\xspace} 
  \newcommand{\ellip}{\ensuremath{X^{\text{ell}}}\xspace} 

  \newcommand{\param}{{\mathchoice{\mkern1mu\mbox{\raise2.2pt\hbox{$
  \centerdot$}}
  \mkern1mu}{\mkern1mu\mbox{\raise2.2pt\hbox{$\centerdot$}}\mkern1mu}{
  \mkern1.5mu\centerdot\mkern1.5mu}{\mkern1.5mu\centerdot\mkern1.5mu}}}

  \AtEndDocument{\bigskip{\footnotesize%
  Talia Fern\'os: \texttt{t\_fernos@uncg.edu}
  \par
  Max Forester: \texttt{forester@math.ou.edu} 
  \par
  Jing Tao: \texttt{jing@math.ou.edu} 
  }}

  \fancypagestyle{plain}

\begin{document}


  \title    {Effective quasimorphisms on right-angled Artin groups}
  \author   {Talia Fern\'os, Max Forester, and Jing Tao}
  \date{}

  \maketitle
  \thispagestyle{empty}

  \begin{abstract} 

    We construct new families of quasimorphisms on many groups acting on
    CAT(0) cube complexes. These quasimorphisms have a uniformly bounded
    defect of $12$, and they ``see'' all elements that act hyperbolically
    on the cube complex. We deduce that all such elements have stable
    commutator length at least 1/24. 

    The group actions for which these results apply include the standard
    actions of right-angled Artin groups on their associated CAT(0) cube
    complexes. In particular, every non-trivial element of a right-angled
    Artin group has stable commutator length at least 1/24. 

    These results make use of some new tools that we develop for the
    study of group actions on CAT(0) cube complexes: the essential
    characteristic set and equivariant Euclidean embeddings. 

  \end{abstract}

\section{Introduction} 
  
  \label{Sec:Intro}

  In this paper, we construct quasimorphisms on groups that admit actions on
  CAT(0) cube complexes. Our emphasis is on finding quasimorphisms that are
  both \emph{efficient} and \emph{effective}. By ``efficient'' we mean that
  the quasimorphisms have low defect. By ``effective'' we mean that the
  quasimorphisms take non-zero values on specified elements of the
  group. These two qualities, taken together, allow one to establish lower
  bounds for stable commutator length ($\scl$) in the group. 

  According to Bavard Duality \cite{Bavard}, if ${\varphi}$ is a
  homogeneous quasimorphism of defect at most $D$ and ${\varphi}(g) \geq
  1$, then $\scl(g) \geq 1/2D$. Thus, for the strongest bound on $\scl$,
  one needs to find effective quasimorphisms with the smallest possible
  defect. 

  The quasimorphisms we define have similarities with the
  ``non-overlapping'' counting quasimorphisms of Epstein and Fujiwara
  \cite{EpsteinFujiwara}, which in turn are a variation of the Brooks
  counting quasimorphisms on free groups \cite{Brooks}. If $X$ is a CAT(0)
  cube complex, there is a notion of a \emph{tightly nested segment} of
  half-spaces in $X$. If $G$ acts on $X$ \emph{non-transversely} (see 
  \defref{NonTransverse2}), then for each tightly nested segment $\g$ there
  is an associated counting quasimorphism $\varphi_\g$. This function
  counts non-overlapping copies (or $G$--translates) of $\g$ and $\bg$
  inside characteristic subcomplexes of elements of $G$. Using the
  \emph{median property} of CAT(0) cube complexes, we show that
  $\varphi_\g$ has defect at most $6$, and therefore its homogenization
  $\widehat{\varphi}_\g$ has defect at most $12$. (Note that this bound is
  independent of both the length of $\g$ and the dimension of $X$.) 

  We now have a large supply of efficient quasimorphisms, but it is by no
  means clear that any of them are non-trivial. Our main task, given an
  element $g \in G$, is to find a tightly nested segment $\g$ such that
  $\widehat{\varphi}_\g(g) \geq 1$. This will only be possible for suitable
  elements $g$; for instance, if $g$ is conjugate to $g^{-1}$, then $\scl(g)
  = 0$ and every homogeneous quasimorphism vanishes on $g$. 

  For our main result we consider cube complexes with group actions that
  have properties in common with the standard actions of right angled Artin
  groups on their associated CAT(0) cube complexes. These are called
  \emph{RAAG-like} actions; see \secref{RAAG-like} and \defref{raaglike}.
  Our main theorem is that for such actions, the desired segments $\g$ can
  be found for \emph{every} hyperbolic element $g$. Using Bavard Duality,
  we obtain:

  \begin{introthm} \label{Thm:Main}

    Let $X$ be a CAT(0) cube complex with a RAAG-like action by $G$. Then
    $\scl(g) \geq 1/24$ for every hyperbolic element $g\in G$. 

  \end{introthm}

  Since the standard action of a right-angled Artin group on its associated
  CAT(0) cube complex is RAAG-like, with all non-trivial elements acting
  hyperbolically, the following corollary is immediate.  
  
  \begin{introcor} \label{Cor:Main}

    Let $G$ be a right-angled Artin group. Then $\scl(g) \geq 1/24$ for
    every nontrivial $g \in G$. 

  \end{introcor}
  
  What is perhaps surprising about this result is that there is a
  \emph{uniform} gap for scl, independent of the dimension of $X$. Note
  that in \thmref{Main} we do not assume that $X$ is either
  finite-dimensional or locally finite; thus \corref{Main} applies to
  right-angled Artin groups defined over arbitrary simplicial graphs.

  The defining properties of RAAG-like actions arose naturally while
  working out the arguments in this paper. It turns out, however, that
  RAAG-like actions are closely related to the \emph{special cube
  complexes} of Haglund and Wise \cite{HaglundWise}. That is, if $G$ acts
  freely on $X$, then the action is RAAG-like if and only if the quotient
  complex $X/G$ is special. See \secref{RAAG-like} and \remref{Special} for
  the precise correspondence between these notions. 

  \begin{introcor} \label{Cor:Main2}

    Let $G$ be the fundamental group of a special cube complex. Then
    $\scl(g) \geq 1/24$ for every non-trivial $g \in G$. 

  \end{introcor}

  This follows from \thmref{Main} since the action of $G$ on the
  universal cover is RAAG-like, with every non-trivial element acting
  hyerbolically. Alternatively, it follows from \corref{Main} and
  monotonicity, since every such group embeds into a right-angled
  Artin group.

  \subsection*{Related results}

  There are other gap theorems for stable commutator length in the
  literature, though in some cases the emphasis is on the existence of a
  gap, rather than its size. The first such result was Duncan and Howie's
  theorem \cite{DuncanHowie} that every non-trivial element of a free group
  has stable commutator length at least $1/2$. In \cite{CFL} it was shown
  that in Baumslag--Solitar groups, stable commutator length is either zero
  or at least $1/12$. A different result in \cite{CFL} states that if $G$
  acts on a tree, then $\scl(g) \geq 1/12$ for every ``well-aligned''
  element $g\in G$. There are also gap theorems for stable commutator
  length in hyperbolic groups \cite{Gromov,CalegariFujiwara} and in mapping
  class groups (and their finite-index subgroups) \cite{BBF2}, where
  existence of a gap is established. In these cases it is also determined
  which elements of the group have positive $\scl$. In
  \cite{CalegariFujiwara}, the size of the gap in the case of a hyperbolic
  group is estimated, in terms of the number of generators and the
  hyperbolicity constant. 

  In \cite[Corollary 6.13]{Koberda}, it was shown that every finitely
  generated right-angled Artin group $G$ embeds into the Torelli subgroup
  of the mapping class group of a surface.  Since scl is positive on the
  Torelli group \cite{BBF2}, monotonicity implies that every non-trivial
  element of $G$ has positive scl. However, the lower bounds obtained in
  this way are neither explicit nor uniform. For instance, the genus of the
  surface needed in \cite{Koberda} grows with the number of generators of
  $G$, and this affects the bounds arising in \cite{BBF2}  (which go to
  zero as the genus grows).  

  There are numerous results on the existence of homogeneous quasimorphisms
  on groups, where the purpose is to show that the group has non-zero
  second bounded cohomology. Let $\widetilde{\QH}(G)$ denote the space of
  homogeneous quasimorphisms on $G$, modulo homomorphisms. Then
  $\widetilde{\QH}(G)$ is a subspace of $H^2_b(G;\R)$. 
  In \cite{EpsteinFujiwara} it was shown that $\widetilde{\QH}(G)$ is
  infinite-dimensional for any hyperbolic group $G$. Recent results in this
  direction, involving both wider classes of groups and more general
  coefficient modules, include \cite{HullOsin} and \cite{BBF1}. 
  If $G$ is a non-abelian right-angled Artin group then 
  $\widetilde{\QH}(G)$ is infinite-dimensional (via a retraction onto a free
  subgroup). The quasimorphisms defined in this paper 
  include an infinite collection that is linearly independent in
  $\widetilde{\QH}(G)$; see \propref{linear-indep}. These appear to be
  different from the quasimorphisms one obtains via retractions. 

  We have mentioned that the \emph{median property} of CAT(0) cube
  complexes is used to control the defect of our quasimorphisms. The use of
  medians in this context originated in \cite{CFI}, where they are used to
  define a bounded cohomology class (the \emph{median class}) which has
  good functorial properties. This class is defined, and is
  \emph{non-trivial}, whenever one has a non-elementary group action on a
  finite-dimensional CAT(0) cube complex. One consequence, among many
  others, is that $H^2_b(G;M)$ is non-trivial for any such group, for a
  suitably defined coefficient module $M$. 

  Our upper bound of $12$ for the defect of the quasimorphisms
  $\widehat{\phi}_\g$ can actually be lowered to $6$ in the special case
  when the CAT(0) cube complex is $1$--dimensional; see \remref{CFL6.6}.
  This statement then coincides with Theorem 6.6 of \cite{CFL}, and thus
  we obtain a new proof of the latter result. 

  While this paper was being considered for publication, Heuer
  \cite{Heuer} announced a sharp lower bound of 1/2 for scl in RAAGs,
  using a very different family of quasimorphisms. 

  \subsection*{Methods}
  
  The fundamental result upon which most of our arguments depend is the
  existence of \emph{equivariant Euclidean embeddings}, proved in
  \propref{TautEmbedding}. To state this result, we first note that every
  element $g \in G$ has a \emph{minimal subcomplex} $M_g \subseteq X$, and if
  $g$ is hyperbolic then this subcomplex admits a $\gen{g}$--invariant
  product decomposition $M_g \cong \mess_g \times \fix_g$. The action of $g$ on
  $\fix_g$ is trivial and every edge in $\mess_g$ is on a combinatorial
  axis for $g$. We call $\mess_g$ the \emph{essential minimal set} for $g$.
  Furthermore, we show that $\mess_g$ is always a finite-dimensional CAT(0)
  cube complex. However, $\mess_g$ is not always a convex subcomplex of
  $X$. We denote by $\ess_g$ its convex hull in $X$ and refer to $\ess_g$
  as the \emph{essential characteristic set} for $g$. The subcomplex
  $\ess_g$ is in general much more complicated than $\mess_g$ and can have
  infinite dimension. In \secref{CharSets}, we give a complete
  characterization of when $\ess_g$ is finite-dimensional and when $\ess_g$
  and $\mess_g$ are the same. 

  \propref{TautEmbedding} states that under suitable assumptions 
  there is a $\gen{g}$--equivariant embedding of
  $\ess_g$ into $\R^d$, where $d = \dim \ess_g$. That is, there is an
  embedding of cube complexes $\ess_g \into \R^d$ such that the action of
  $\gen{g}$ on $\ess_g$ extends to an action on $\R^d$ (preserving its
  standard cubing). Furthermore, the embedding induces a bijection between
  the half-spaces of $\ess_g$ and those of $\R^d$.

  It is well known that any interval in a CAT(0) cube complex
  admits an embedding into $\R^d$ for some $d$. This result is proved using
  Dilworth's theorem on partially ordered sets of finite width; see
  \cite{BCGNW} for details. What is new in our result is the
  equivariance. In order to prove it, we first state and prove an
  equivariant version of Dilworth's theorem, \lemref{EquiDilworth}. 

  An important aspect of the equivariant Euclidean embedding is that it
  provides a geometric framework for understanding the fine structure of
  the set of half-spaces of $\ess_g$, considered as a partially ordered
  set. This set becomes identified with the set of half-spaces of $\R^d$,
  and the partial ordering from $\ess_g$ is determined by the knowledge of
  which cubes in $\R^d$ are occupied by $\ess_g$ (cf. \remref{Square}).
  Tools such as the Quadrant Lemma and the Elbow Lemma (see
  \secref{Quadrants}) can be used to retrieve information about the partial
  ordering. These tools become available once $\ess_g$ has been embedded
  into $\R^d$.  

  \subsection*{An outline of the paper}

  In \secref{Prelim} we present background on several topics, including
  quasimorphisms and stable commutator length, CAT(0) cube complexes,
  and right-angled Artin groups. 

  In \secref{CharSets} we define the \emph{essential minimal set} and
  the \emph{essential characteristic set}, and establish their
  properties. We determine when they agree, and when the latter has
  finite dimension. 

  In \secref{Non-transverse} we define \emph{non-transverse}
  actions. For such actions we also define the quasimorphisms $\psi_\g$ and
  $\varphi_\g$ and establish the bounds on defect, using medians. We
  construct an infinite linearly independent set in
  $\widetilde{\QH}(A_{\Gamma})$, for any non-abelian right-angled
  Artin group $A_{\Gamma}$. 

  In \secref{Dilworth} we prove the equivariant Dilworth theorem, and apply
  it to prove the existence of equivariant Euclidean embeddings of
  essential characteristic sets. 

  In \secref{Quadrants} we introduce \emph{quadrants} and prove two basic
  results, the Quadrant Lemma and the Elbow Lemma. These are the primary
  tools used for studying the essential characteristic set $\ess_g$ once it
  has been equivariantly embedded into $\R^d$. 

  In \secref{RAAG-like} we discuss \emph{RAAG-like} actions on CAT(0) cube
  complexes. 

  In Sections \ref{Sec:TightlyNested} and \ref{Sec:Last} we carry out the
  rather intricate arguments needed to show that $\widehat{\varphi}_\g(g)
  \geq 1$ for the appropriate choice of $\g$. Essentially all of the
  effort in these sections is devoted to showing that $\ess_g$ contains
  no $G$--translate of $\bg$.

  \subsection*{Acknowledgments}

  Fern\'os was partially supported by NSF award DMS-1312928, Forester by
  NSF award DMS-1105765, and Tao by NSF awards DMS-1311834, DMS-1611758,
  DMS-1651963. The authors thank the referee for many helpful comments
  and for pointing out an error in the original proof of \propref{CAT0}.  

\section{Preliminaries} 

  \label{Sec:Prelim} 

  In this section we establish notation and background for the rest of the
  paper. We start with the topics of quasimorphisms and stable commutator
  length. For more detail see \cite{scl}. Then we give some background on
  CAT(0) cube complexes, focusing on the structure of their half spaces and
  their median structure. More information on these topics can be found in
  \cite{Sageev,Roller,Haglund,ChatterjiNiblo,Nica}. The section
  concludes with a brief overview of right-angled Artin groups and
  properties of their associated CAT(0) cube complexes. These
  properties lead to the notion of \emph{RAAG-like} actions, to be
  defined in \secref{RAAG-like}. 

  \begin{notation}

    Throughout the paper we use the symbols ``$\subset$'' and
    ``$\supset$'' to denote \emph{strict} inclusion only. 

  \end{notation}

  \subsection*{Quasimorphisms and stable commutator length}
  
  Let $G$ be any group. A  map $\varphi \from G \to \R$ is a
  \emph{quasimorphism} on $G$ if there is a constant $D \ge 0$ such that
  for all $g,h \in G$, 
  \[ 
    \abs{\varphi(gh) - \varphi(g) - \varphi(h)} \le D.
  \] 
  The smallest $D$ that satisfies the inequality above is called the
  \emph{defect} of $\varphi$. It is immediate that a quasimorphism is a
  homomorphism if and only if its defect is $0$. 

  A quasimorphism $\varphi$ is \emph{homogeneous} if $\varphi(g^n) = n
  \varphi(g)$ for all $g \in G$ and $n \in \Z$. Given any quasimorphism
  $\varphi$, its \emph{homogenization} $\widehat{\varphi}$ is defined by
  \[ 
    \widehat{\varphi}(g) = \lim_{n \to \infty} \frac{\varphi(g^n)}{n}. 
  \] 
  It is straightforward to check $\widehat\varphi$ is a homogeneous
  quasimorphism. Its defect can be estimated as follows: 

  \begin{lemma} \label{Lem:Homogenization}

    If $\varphi$ is a quasimorphism of defect at most $D$, then its
    homogenization has defect at most $2D$. 
    
  \end{lemma}
  
  Two maps $\varphi, \psi \from G \to \R$ are uniformly close if there 
  exists $D \ge 0$ such that $\abs{\varphi(g) - \psi(g)} \le D$ for all $g
  \in G$. It is easy to check that any map uniformly close to a
  quasimorphism is a quasimorphism. Further, the following statement
  holds: 

  \begin{lemma} \label{Lem:Bounded} 

    If $\varphi$ is uniformly close to a quasimorphism $\psi$, then
    $\widehat{\varphi} = \widehat{\psi}$. 
    
  \end{lemma}
  
  \begin{proof}
    
    By assumption, there exists $D \ge 0$ such that
    $\abs{\varphi(g)-\psi(g)} \le D$ for all $g \in G$. Then 
    \[ 
      \abs{\widehat{\varphi}(g) - \widehat{\psi}(g)} 
      = \left\vert \lim_{n \to \infty} \frac{\varphi(g^n)}{n} -
        \lim_{n \to \infty} \frac{\psi(g^n)}{n} \right\rvert  
      = \lim_{n \to \infty} \frac{\abs{\varphi(g^n) -
      \psi(g^n)}}{n} 
      \le \lim_{n \to \infty} \frac{D}{n} = 0. \qedhere
    \] 
  \end{proof}
  
  Now denote by $[G,G]$ the commutator subgroup of $G$. Given an element $g
  \in [G,G]$, the \emph{commutator length} $\text{cl}(g)$ of $g$ is the
  minimal number of commutators whose product equals $g$. The commutator
  length of the identity element is $0$. For any $g \in [G,G]$, the
  \emph{stable commutator length} of $g$ is 
  \[ 
    \scl(g) = \lim_{n \to \infty} \frac{\text{cl}(g^n)}{n}.
  \] 
  Note that $\scl(g^n) = n \scl(g)$ for all $n\in \Z$ and $g\in G$. This
  formula allows one to define $\scl$ for elements that are only
  virtually in $[G,G]$. By convention, $\scl(g) = \infty$ if no power of 
  $g$ lies in $[G,G]$. 
  
  The relationship between stable commutator length and quasimorphisms on
  $G$ is expressed by Bavard duality \cite{Bavard}. We state the
  easier direction below: 
  
  \begin{lemma}[Easy direction of Bavard Duality] \label{Lem:SCL}

    For any $g \in [G,G]$, if $\varphi$ is a homogeneous quasimorphism
    on $G$ with defect at most $D$, then 
    \[ 
      \scl(g) \ge \frac{\varphi(g)}{2D}.
    \] 
    
  \end{lemma}
  
  \subsection*{CAT(0) cube complexes}
  
  A cube of dimension $d$ is an isometric copy of $[0,1]^d$ with the
  standard Euclidean metric. A face of a cube is obtained by fixing any
  number of coordinates to be $0$ or $1$. This is naturally a cube of the
  appropriate dimension. A \emph{midcube} is the subset of the cube
  obtained by fixing one of the coordinates to be $1/2$. 
  
  A \emph{cube complex} $X$ is a space obtained from a collection of cubes
  with some faces identified via isometries. The dimension of $X$ is the
  dimension of a maximal dimensional cube if it exists; otherwise the
  dimension of $X$ is infinite. We equip $X$ with the path metric induced
  by the Euclidean metric on each cube. By Gromov's link condition, $X$ is
  non-positively curved if and only if the link of every vertex of $X$ is a
  flag complex. A cube complex $X$ is CAT(0) if and only if it is
  non-positively curved and simply connected.

  Let $X$ be a CAT(0) cube complex. By an \emph{edge path} of length $n$ we
  will mean a sequence of vertices $x_0,\ldots,x_n$, such that adjacent
  vertices $x_i$ and $x_{i+1}$ are joined by an edge of $X$. If
  $p=x_0,\ldots,x_n$ and $q=y_0,\ldots,y_m$ are two edge paths with $x_n =
  y_0$, then their concatenation is the edge path $p \cdot q =
  x_0,\cdots,x_n, y_1, \cdots y_m$. 

  We will ignore the CAT(0) metric on $X$ and consider the
  \emph{combinatorial metric} on its vertex set, which measures
  distance $d(x,y)$ between two vertices $x$ and $y$ as the minimal
  length of an edge path joining them. An edge path from $x$ to $y$ is a
  \emph{geodesic} if it has length $d(x,y)$. An infinite sequence of
  vertices in $X$ is a geodesic if every finite consecutive subsequence
  is a geodesic. 
  
  A \emph{hyperplane} in $X$ is a connected subset whose intersection with
  each cube of $X$ is either empty or is a midcube. This set always divides
  $X$ into two disjoint components. The closure of a component is called a
  \emph{half-space} $H$ of $X$. The closure of the other component is
  denoted by $\bH$. We denote by $\partial H$ the boundary hyperplane of
  $H$ and note that $\partial H = \partial \bH$. 
  
  A subcomplex $C \subseteq X$ is \emph{convex} if every geodesic in $X$
  between two of its vertices is contained entirely in $C$. If $Y \subseteq
  X$ is any subcomplex, the \emph{convex hull} $C(Y)$ of $Y$ is the
  smallest convex subcomplex containing $Y$. Equivalently, it is the
  largest subcomplex of $X$ that is contained in the intersection of all
  half-spaces containing $Y$. 

  For any vertices $x,y \in X$, we will denote by $C(x,y)$ the convex hull
  $C \left( \set{x,y} \right)$. 

  A hyperplane $\partial H$ is \emph{dual} to an edge (or vice versa) if
  $\partial H$ intersects the edge. A half-space $H$ is dual to an edge if 
  $\partial H$ is. A cube $C$ is dual to a hyperplane $\partial H$ if $C$
  contains an edge dual to $\partial H$. The \emph{neighborhood} of
  $\partial H$ is the union $N(\partial H)$ of all cubes dual to $\partial
  H$. By \cite[Theorem 2.12]{Haglund}, $N(\partial H)$ is convex. 
  Further, there is
  a an involution on $N(\partial H)$ that fixes $\partial H$ pointwise
  and swaps the endpoints of each edge dual to $\partial H$ (in fact,
  $N(\partial H) \cong \partial H \times [0,1]$). 

  Let $\half(X)$ be the collection of half-spaces of $X$. This is 
  partially ordered by inclusion. We say two half-spaces are
  \emph{nested} if they are linearly ordered; they are \emph{tightly
  nested} if they are nested and there is no third half-space that lies
  properly between them. The map $\half(X) \to \half(X)$ sending $H$
  to $\bH$ is an order-reversing involution. 

  Two half-spaces $H, H'$ of $X$ are \emph{transverse}, denoted by $H
  \trans H'$, if all four intersections 
  \[ 
    H \cap H', \qquad H \cap \bH', \qquad \bH \cap H, \qquad \bH \cap
    \bH',
  \] 
  are non-empty. When this happens, then there is a cube $C$ in $X$
  such that $\partial H \cap C$ and $\partial H' \cap C$ are different
  midcubes of $C$. More generally, if $H_1,\ldots,H_n$ are pairwise
  transverse, then there is a cube $C$ in $X$ of dimension $n$ such
  that $\partial H_1 \cap C, \ldots,\partial H_n \cap C$ are the $n$
  midcubes of $C$. 
  
  Given two vertices $x,y \in X$, the \emph{interval between $x$ and $y$}
  is 
  \[ 
    [x,y] = \set{H \in \half \st y\in H, x\in \bH}.
  \] 
  Two distinct half-spaces $H, H' \in [x,y]$ are always either nested or
  transverse. The interval $[y,x]$ is exactly the set of half spaces
  $\set{ \bH \st H \in [x,y]}$. 

  An \emph{oriented edge} $e = (x,y)$ is an edge whose vertices $x, y$
  have been designated as \emph{initial} and \emph{terminal}
  respectively. Given an edge path $x_0, \dotsc, x_n$, each edge $(x_i,
  x_{i+1})$ receives an induced orientation with $x_i$ initial and
  $x_{i+1}$ terminal. For any oriented edge $e = (x,y)$, the \emph{half
  space dual to} $e$ is the unique half-space in the interval $[x,y]$; it
  is dual to $e$ considered as an unoriented edge, and it contains $y$
  but not $x$. 

  An edge path is a geodesic if and only if it crosses no hyperplane twice.
  Two geodesics from $x$ to $y$ determine the same set of half-spaces
  $[x,y]$, and every half-space $H \in [x,y]$ is dual to some edge on every
  geodesic from $x$ to $y$. Therefore, the combinatorial distance $d(x,y)$
  is the same as the cardinality of $[x,y]$. See \cite[Theorem
  4.13]{Sageev} for more details.
  
  \subsection*{Ultrafilters}

  Suppose $\sigma$ is a function assigning to each hyperplane $h$ in $X$ a
  half-space $H$ with $\partial H = h$. Then $\sigma$ is an
  \emph{ultrafilter} if $\sigma(h)$ and $\sigma(h')$ have non-trivial
  intersection for every pair of hyperplanes $h, h'$. An alternative
  viewpoint is to simply specify the image of $\sigma$, as a subset of
  $\half(X)$ that contains exactly one half-space from each pair $\set{H,
  \bH}$, such that no two elements are disjoint. For this reason, $\sigma$
  is sometimes called an ultrafilter ``on $\half(X)$". 

  For each vertex $v$ of $X$ there is a \emph{principal ultrafilter} of
  $v$, defined by choosing $\sigma(h)$ to be the half-space with boundary
  $h$ containing $v$. Neighboring vertices define principal ultrafilters
  that differ on a single hyperplane (the one that is dual to the edge
  separating the vertices). Conversely, if two principal ultrafilters
  differ on a single hyperplane, then the corresponding vertices bound an
  edge, dual to that hyperplane. Since $X$ is connected, any two principal
  ultrafilters will differ on finitely many hyperplanes. Indeed, the number
  of such hyperplanes is precisely the distance between the two vertices. 

  The principal ultrafilters admit an intrinsic characterization: an
  ultrafilter $\sigma$ on $\half(X)$ is principal if and only if
  it satisfies the descending chain condition: whenever $\{h_i\}$ is a
  sequence of hyperplanes such that $\sigma(h_i) \supseteq
  \sigma(h_{i+1})$ for all $i$, the sequence is eventually
  constant. It follows that if an ultrafilter differs
  from a principal one on finitely many hyperplanes, it will also be
  principal. 

  Knowledge of the principal ultrafilters on $\half(X)$ completely
  determines $X$ as a CAT(0) cube complex. The \emph{Sageev construction}
  is the name for the process of building a cube complex from its partially
  ordered set of half-spaces. The $1$--skeleton of $X$ is determined from
  principal ultrafilters as already described, and cubes are added whenever
  their $1$--skeleta are present \cite{Sageev}. 

  More generally, let $\half$ be any partially ordered set with an
  order-reversing free involution $H \mapsto \bH$, such that every interval
  is finite. The Sageev construction yields a CAT(0) cube complex
  $X(\half)$ whose half-spaces correspond to $\half$ as a partially ordered
  set with involution \cite{Roller}. It is often convenient to think of
  vertices of $X$ as principal ultrafilters, and to identify $X$ with the
  result of the Sageev construction performed on $\half(X)$. 

  \subsection*{Medians}

  Given three vertices $x, y, z\in X$, there is a unique vertex $m = m(x,
  y, z)$ called the \emph{median} such that $[a, b] = [a, m] \cup [m, b]$
  for all pairs $\set{a,b} \subset \set{x,y,z}$. For completeness we sketch
  the proof, since the standard reference \cite{Roller} is unpublished. 

  As an ultrafilter, $m$ is defined by simply assigning to each hyperplane
  the half-space which contains either two or three of the vertices
  $\set{x, y, z}$. Two such half-spaces cannot be disjoint, so this rule
  does indeed define an ultrafilter. This ultrafilter is principal (i.e. it
  defines a \emph{vertex}) because it differs from the principal
  ultrafilter of $x$ on finitely many hyperplanes: if $H$ is chosen by $m$
  and $x\not\in H$, then $y,z \in H$; hence $H \in [x, y] \cap [x,z]$, a
  finite set. Finally, given $a, b \in \set{x, y, z}$, every half-space
  containing $a$ and $b$ also contains $m$, by definition. Thus, no
  hyperplane can separate $m$ from $a$ and $b$, and therefore $[a, b] =
  [a,m] \cup [m,b]$. 
  
  A vertex $z$ lies on a geodesic edge path from $x$ to $y$ if and only if
  $z = m(x, z, y)$. Therefore, $z \in C(x,y)$ if and only $z = m(x,z,y)$. 
  
  \subsection*{Segments}

  By a \emph{segment} \g of length $n$ we will mean a chain of half-spaces
  $H_1 \supset H_2 \supset \cdots \supset H_n$ such that $H_i$ and
  $H_{i+1}$ are tightly nested for all $i=1,\ldots,n-1$. The
  \emph{inverse} of \g is the segment \bg: $\bH_n \supset \bH_{n-1}
  \supset \cdots \supset \bH_1$.  

    Let $\g$ and $\g'$ be segments. We write $\g > \g'$ if every
    half-space in $\g$ contains every half-space in $\g'$. We say that
    $\g$ and $\g'$ are \emph{nested} if either $\g > \g'$ or $\g' > \g$. 
  
  \begin{definition}
    
    Two segments $\g$ and $\g'$ are said to  \emph{overlap} if
    either $\g \cap \g' \neq \emptyset$ or there exist $H \in \g$ and
    $H' \in \g'$ with $H \trans H'$. Otherwise, they are
    \emph{non-overlapping}. 
  
  \end{definition}
  
  \begin{lemma} \label{Lem:Non-Over}
    
    Suppose $\g_1$ and $\g_2$ are non-overlapping segments that are
    contained in $[x,y]$. Then $\g_1$ and $\g_2$ are nested.
    
  \end{lemma}
  
  \begin{proof}
   
    As mentioned above, any two half-spaces in $[x,y]$ are either
    nested or transverse. Therefore, since $\gamma_1$ and $\gamma_2$ 
    are non-overlapping, their union is linearly ordered by
    inclusion. The result follows, since each $\gamma_i$ is a
    segment. \qedhere  

  \end{proof}
  
  \subsection*{Right-angled Artin groups}
  
  Let $\Gamma$ be a simplicial graph (i.e. a simplicial complex of
  dimension at most $1$), with vertex set $V(\Gamma)$ and 
  edge set $E(\Gamma)$. The \emph{right-angled Artin group} $A_{\Gamma}$
  is defined to be the group with generating set $V(\Gamma)$ and relations
  $\set{[v,w] \st \set{v,w} \in E(\Gamma)}$. That is, two generators
  commute if and only if they bound an edge in $\Gamma$, and there are no
  other defining relations. 

  There is a naturally defined non-positively curved cube complex
  which is a $K(A_{\Gamma},1)$ complex, obtained as a union
  of tori corresponding to complete subgraphs of $\Gamma$ (see Davis
  \cite[11.6]{Davis}, for example). The universal cover $X_{\Gamma}$ is a
  CAT(0) cube complex with a free action by $A_{\Gamma}$. The oriented
  edges of $X_{\Gamma}$ can be labeled by the generators of $A_{\Gamma}$
  and their inverses in a natural way: each such edge is a lift of a loop
  representing that generator (or its inverse). 

  This labeling has the property that two oriented edges are in the same
  $A_{\Gamma}$--orbit if and only if their labels agree. Also, the oriented
  edges that are dual to any given half-space will always have the same
  label, so the label may be assigned to the half-space
  itself. Half-spaces in the same $A_{\Gamma}$--orbit will have the same
  label. 

  The half-space labels lead to several useful observations. Each
  $2$--cell of $X_{\Gamma}$ is a square whose boundary is labeled by a
  commutator $[v,w]$, with $v \not= w$. It follows that no two
  half-spaces with the same label can be transverse in $X_{\Gamma}$.
  Since no label equals its inverse, $A_{\Gamma}$ acts without inversion
  on $X_{\Gamma}$. Some additional properties of $X_{\Gamma}$ related to
  the edge-labeling will be discussed in \secref{RAAG-like}.

\section{Automorphisms and characteristic sets} 
  
  \label{Sec:CharSets}

  In this section we discuss automorphisms of CAT(0) cube complexes
  and their characteristic sets. We define the \emph{essential
  characteristic set} and the \emph{essential minimal set} of a
  hyperbolic automorphism, and we determine the structures of these
  sets. The latter set is always finite-dimensional, whereas the
  former is a subcomplex which plays an essential role throughout the
  paper. Toward the end of the section, we characterize when these
  sets agree (\propref{Equality}) and when the essential
  characteristic set is finite-dimensional (\corref{XFiniteDim}). 

  \subsection*{Basic notions}

  Following Haglund \cite{Haglund}, an automorphism $g$ of a CAT(0) cube
  complex $X$ \emph{acts with inversion} if there is a half-space $H$
  such that $g(H) = \bH$. When this occurs, $g$ stabilizes the hyperplane
  $\partial H$. For any automorphism $g$ of $X$, the action of $g$ on the
  cubical subdivision of $X$ is always \emph{without} inversion. Note,
  however, that for some of our results, we will not be free to perform
  this modification; see \remref{Subdivision}. 
  
  For an automorphism $g$ of $X$, the \emph{translation distance} of $g$ is
  $\ell_g = \min_{x \in X}d(x,gx)$, where $x$ ranges over the vertices of
  $X$. If $g$ and all of its powers act without inversion, we say that $g$ is
  \emph{hyperbolic} if $\ell_g>0$ and \emph{elliptic} otherwise. Haglund
  showed that when $g$ is hyperbolic, there is an infinite combinatorial
  geodesic in $X$ that is preserved by $g$, on which $g$ acts as a
  translation of magnitude $\ell_g$. Any such geodesic will be called a
  \emph{combinatorial axis} for $g$. It has a natural orientation,
  relative to which the translation by $g$ is in the \emph{forward}
  direction. Note that $g$ and $g^{-1}$ have the same combinatorial axes,
  but they determine opposite orientations. 

  Haglund also showed that any two combinatorial axes for $g$ cross the
  same hyperplanes, in the same directions. That is, the set
  of half-spaces that are dual to oriented edges in any axis is
  independent of the choice of axis. We define the \emph{positive
    half-space axis of $g$}: 
  \[ 
    \axis_g^+ \ = \ \set{H \in \half(X) \st H \text{ is dual to a
    positively oriented edge in a combinatorial axis for } g}. 
  \]
  We also define the \emph{negative half-space axis} $\axis_g^- = \set{\bH
  \st H \in \axis_g^+}$; note that $\axis_g^- = \axis_{g^{-1}}^+$. The
  \emph{full half-space axis} is $\axis_g = \axis_g^+ \sqcup \axis_g^-$. 

  If $L$ is a combinatorial axis for $g$, then for every $H \in \axis_g^+$,
  the intersection $L \cap H$ is a ray containing the attracting end of $L$
  (since $L$ crosses $\partial H$ exactly once). Note also that $gH \not=
  H$ for all $H \in \axis_g$, for otherwise $g$ would fix the unique edge
  of $L$ dual to $H$, contradicting hyperbolicity of $g$.

  \begin{remark}\label{Rem:TransNest}

    For any distinct half-spaces $H,H' \in \axis_g^+$, either $H\trans
    H'$, $H\subset H'$, or $H'\subset H$. The only other possibilities
    are that $\bH' \subset H$ or $\bH \subset H'$, but these imply
    that either $H
    \cap H'$ or $\overline{H} \cap \overline{H}'$ is empty. But every
    combinatorial axis for $g$ meets both of these sets in an infinite
    ray. 

  \end{remark}
    
  \begin{remark}\label{Rem:TransNest2}

    For any $H \in \axis_g^+$ and 
    $n > 0$, if $H$ and $g^n H$ are not transverse, then $H \supset g^nH$.
    To see this, let $L$ be any oriented combinatorial axis for $g$. Let $e
    = (x,y)$ be the oriented edge on $L$ dual to $H$. Then $e$ lies on a
    geodesic edge path from $x$ to $g^nx$. In other words, $H \in [x,g^n x]$, 
    and so $g^n x \in H$. Since $x \notin H$, $g^nx \notin g^n H$. It
    follows that $H \supset g^nH$ (rather than $H \subset g^nH$). 

  \end{remark}

  Let $G$ be a group acting on $X$ by automorphisms. We will always assume
  (here and for the rest of the paper) that all elements of $G$ act without
  inversion. Under this assumption, Haglund showed that every element $g
  \in G$ is either elliptic or hyperbolic. 
  
  \subsection*{The minimal set, the characteristic set, and their
    product decompositions}

  \begin{definition}
    
    For any $g \in G$, the \emph{minimal set} of $g$ is the full 
    subcomplex $M_g \subseteq X$ generated by the vertices of $X$ that
    realize the translation distance of $g$. 
    
  \end{definition}

  Since $g$ and all of its powers act without inversion, there are two
  types of behavior for $M_g$. If $g$ is elliptic then $M_g$ is the
  subcomplex of fixed points of $g$. If $g$ is hyperbolic then $M_g$ is the
  smallest full subcomplex containing all combinatorial axes for $g$. It
  is non-empty, and \emph{every} vertex of $M_g$ is on a combinatorial
  axis, by \cite[Corollary 6.2]{Haglund}. 
  
  Next we define three more sets of half-spaces when $g\in G$ is
  hyperbolic: 
  \begin{align*}
    S_g \ &= \ \set{H \in \half \st H \text{ contains every combinatorial
          axis of } g } \\
    &= \ \set{H \in \half \st H \text{ contains } M_g}, \\
    \overline{S}_g \ &= \ \set{H \in \half \st H \not\in \axis_g \text{
                     and $H$ contains no combinatorial axis of } g} \\
    &= \ \set{H \in \half \st \bH \in S_g}, \\
    T_g \ &= \ \set{ H \in \half \st H \not\in \axis_g \text{ and
          $\partial H$ separates two combinatorial axes of } g}.    
  \end{align*}
  Recall that the half-spaces \emph{not} in $\axis_g$ are exactly
  those whose boundary hyperplanes do not cross any axis. Thus the
  aforementioned sets define a partition of $\half(X)$: 
  \[ 
    \half(X) \ = \ \axis_g \sqcup S_g \sqcup \overline{S}_g \sqcup
    T_g.
  \] 

  \begin{remark}

    For any group $\Gamma$ acting on $X$, Caprace and Sageev have defined
    a decomposition of $\half(X)$ into \emph{$\Gamma$--essential}, 
    \emph{$\Gamma$--half-essential}, and \emph{$\Gamma$--trivial}
    half-spaces \cite{CapraceSageev}. It can be shown that when $\Gamma =
    \gen{g}$ (with $g$ hyperbolic), these three collections of
    half-spaces coincide with $\axis_g$, $(S_g \cup \overline{S}_g)$, and
    $T_g$, respectively. 

    Using this perspective, some of the results below can be derived
    from results in \cite{CapraceSageev} and \cite{CFI}. Specifically,
    \lemref{CharProduct} is observed in Remark 3.4 of
    \cite{CapraceSageev}, and \lemref{Crossing} can be derived from Lemma
    2.6 of \cite{CFI} (see also \cite[Remark 2.11]{Fernos}). 

    For completeness, we include elementary proofs of these results,
    using the definitions of $\axis_g$, $S_g$, $\overline{S}_g$, and
    $T_g$ given above. 

  \end{remark}

  \begin{lemma}

    Suppose $g\in G$ is hyperbolic. If $H \in \axis_g$ and $K \in T_g$
    then $H \trans K$. 

  \end{lemma}

  \begin{proof}

    Let $L$, $L'$ be combinatorial axes of $g$ such that $L
    \subset K$ and $L' \subset \bK$. Every axis meets both $H$ and
    $\bH$. Thus all four intersections $K \cap H$, $K \cap \bH$, $\bK
    \cap H$, and $\bK \cap \bH$ are non-empty. \qedhere 

  \end{proof}

  \begin{definition}

    If $g\in G$ is hyperbolic, the \emph{characteristic set} of $g$ is
    the convex hull of $M_g$, denoted $X_g$. Equivalently, $X_g$ is
    the largest subcomplex of $X$ contained in $\bigcap _{H \in S_g}
    H$. 

  \end{definition}
  
  The collections of half-spaces $\axis_g$ and $T_g$ define CAT(0)
  cube complexes $\ess_g = X(\axis_g)$ and $\ellip_g = X(T_g)$ by the
  Sageev construction, called the \emph{essential characteristic set}
  and the \emph{elliptic factor} respectively. 

  \begin{lemma} \label{Lem:CharProduct}

    Suppose $g \in G$ is hyperbolic. Then there is a $\gen{g}$--equivariant
    isomorphism of cube complexes $X_g \cong \ess_g \times \ellip_g$. 

  \end{lemma}

  \begin{proof}

    First note that since $\axis_g \trans T_g$, there is an isomorphism
    $\ess_g \times \ellip_g \cong X(\axis_g \cup T_g)$, by \cite[Lemma
    2.5]{CapraceSageev}. We shall define an embedding $X(\axis_g \cup T_g)
    \into X$ and show that its image is $X_g$. 

    The map is defined by extending each principal ultrafilter on $\axis_g
    \cup T_g$ to an ultrafilter on $\half(X)$ by including every half-space
    in $S_g$. These half-spaces have non-trivial intersection with every
    half-space in $\axis_g \cup T_g$, and also with each other, so this
    rule does indeed define an ultrafilter. Moreover, no half-space in
    $S_g$ is contained in any half-space of $\axis_g \cup T_g$, and
    $S_g$ itself contains no descending chains: for each $H\in S_g$,
    $\partial H$ is separated from $M_g$ by only finitely many
    hyperplanes. Therefore the
    descending chain condition is still satisfied. Thus, each vertex of
    $X(\axis_g \cup T_g)$ is mapped to a vertex of $X$. It is clear that
    adjacent vertices map to adjacent vertices, so the map is an embedding
    of cube complexes. 

    Next, the vertices of $X_g$ are exactly the vertices whose principal
    ultrafilters include all half-spaces of $S_g$. These are exactly the
    vertices in the image of our map, so this image is $X_g$. 

    Equivariance holds because the $\gen{g}$--actions on $X(\axis_g)$,
    $X(T_g)$, and $X = X(\half(X))$ are all simultaneously induced by the
    action of $\gen{g}$ on the half-spaces of $X$. \qedhere 

  \end{proof}

  The next result concerns crossing of half-spaces of $\axis_g$. Namely,
  two such half-spaces cross in $\ess_g$ if and only if they cross in
  $X$:

  \begin{lemma} \label{Lem:Crossing}

    Suppose $g \in G$ is hyperbolic. If $H, H' \in \axis_g$ and $H \trans
    H'$ in $X$, then $H \trans H'$ in $\ess_g$. That is, there is a square
    $S \subset \ess_g$ containing edges $e$, $e'$ that are dual to $H$ and
    $H'$ respectively. 

  \end{lemma}

  \begin{proof}

    Recall that $\ess_g = X(\axis_g)$. We
    may embed $\ess_g$ as a convex subcomplex of $X$ in such a way
    that the induced map on half-spaces $\axis_g \to \half(X)$ is
    inclusion; this follows from \lemref{CharProduct}, by choosing a
    vertex $v \in \ellip_g$ and identifying $\ess_g$ with $\ess_g
    \times \set{v}$ in $X_g$. 

    There is a combinatorial retraction $X \to \ess_g$ defined in terms of
    ultrafilters by restriction: each principal ultrafilter on $\half(X)$
    is sent to its intersection with $\axis_g$. The resulting ultrafilter
    still satisfies the descending chain condition, and therefore defines a
    vertex in $\ess_g$. Two adjacent vertices of $X$ will either map to
    adjacent vertices or to the same vertex. This map extends to cubes, and
    each cube maps to a cube in $\ess_g$ by a coordinate projection. More
    specifically, an edge in $X$ is collapsed if and only if its dual
    half-spaces are not in $\axis_g$. It follows that if a square in $X$ is
    dual to two half-spaces in $\axis_g^+$, then its image in $\ess_g$ is
    also a square, dual to the same two half-spaces. Thus, if $H, H' \in
    \axis_g^+$ are transverse in $X$, they are transverse in $\ess_g$.
    \qedhere 

  \end{proof}

  Next we continue to examine the structure of $X_g$. 

  \begin{lemma}

    Suppose $g \in G$ is hyperbolic. Then $g$ acts as an elliptic
    automorphism of $\ellip_g$. 

  \end{lemma}

  \begin{proof}

    If not, any axis $L$ of $g$ acting on $X_g = \ess_g \times \ellip_g$
    would project onto an axis in $\ellip_g$, and $L$ would then cross a
    hyperplane bounding a half-space in $T_g$. However, no axis of $g$
    crosses such a hyperplane. \qedhere 

  \end{proof}

  Accordingly, there is a non-empty subcomplex $\fix_g \subseteq \ellip_g$
  consisting of the fixed points of the $\gen{g}$--action on $\ellip_g$. It
  is a subcomplex because $\gen{g}$ acts without inversion. 

  \begin{lemma}

    Suppose $g \in G$ is hyperbolic. Then there is a
    $\gen{g}$--invariant subcomplex $\mess_g \subseteq \ess_g$ 
    such that $M_g = \mess_g \times \fix_g$ under the identification
    of $X_g$ with $\ess_g \times \ellip_g$. 

  \end{lemma}

  The subcomplex $\mess_g$ is called the \emph{essential minimal set} for
  $g$. 

  \begin{proof}

    If $x$ is a vertex of $M_g$ then no half-space of $T_g$ separates $x$
    from $gx$, since $x$ and $gx$ are on a combinatorial axis. Thus the
    principal ultrafilters at $x$ and at $gx$ agree on half-spaces in
    $T_g$. That is, $g$ fixes the second coordinate of $x$ in $\ess_g
    \times \ellip_g$. Therefore $M_g \subseteq \ess_g \times \fix_g$. 

    Let $\mess_g$ be the projection of $M_g$ onto the first factor of
    $\ess_g \times \ellip_g$, so $M_g \subseteq \mess_g \times \fix_g$. Since
    $\gen{g}$ acts trivially on $\fix_g$, any two vertices of $\mess_g
    \times \fix_g$ with the same first coordinate are moved the same
    distance by $g$. It follows that every vertex of $\mess_g \times
    \fix_g$ is moved distance $\ell_g$, and hence is in $M_g$. Since $M_g$
    is the full subcomplex spanned by its vertices, we have $M_g = \mess_g
    \times \fix_g$. Further, $\gen{g}$--invariance of $\mess_g$ is clear, because
    both $M_g$ and its product structure are $\gen{g}$--invariant. \qedhere

  \end{proof}

  \begin{lemma} \label{Lem:EdgeAxis}

    Suppose $g \in G$ is hyperbolic. Let $e$ be an edge of $M_g$ which
    projects to an edge in the factor $\mess_g$. Then $e$ is on a
    combinatorial axis of $g$. 

  \end{lemma}

  \begin{proof}

    We prove the contrapositive. 
    Let $e = (x,y)$ where $x$ and $y$ are vertices of $M_g$. If $e$ is not
    on any combinatorial axis, then $y$ is not on any geodesic from $x$ to
    $gx$, so $y \not= m(x,y,gx)$. There must be a half-space containing $y$
    but not $x$ or $gx$. The half-space $H$ dual to $e$ is the only
    possibility, since $[x,y] = \set{H}$. 

    Similarly, $x$ is not on any geodesic from $y$ to $gy$, so there must
    be a half-space containing $x$ but not $y$ or $gy$. This can only be
    $\bH$, since $[y,x] = \set{\bH}$. 

    Thus $\partial H$ separates $gx$ from $gy$, and hence is dual to
    $ge$; therefore $g \partial H = \partial H$. Since $g$ is not an
    inversion, we have that $gH = H$. Thus $H \not\in \axis_g$ and $e$
    does not project to an edge in $\mess_g$. \qedhere 

  \end{proof}

  \subsection*{Relationship between $\ess_g$ and $\mess_g$} 

  The convex hull of $\mess_g \subseteq \ess_g$ is $\ess_g$. To see this,
  note that every edge of $\mess_g$ is dual to a half-space of $\ess_g$, by
  \lemref{EdgeAxis}; and conversely, every half-space in $\axis_g^+$ is
  dual to an edge in an axis, and hence to an edge in $\mess_g$. Hence no
  half-space of $\ess_g$ contains $\mess_g$, and therefore $C(\mess_g)$ is
  the intersection of the empty set of half-spaces of $\ess_g$. 

  In this section, we will establish the basic structure of $\mess_g$ and
  address when $\mess_g$ and $\ess_g$ are the same.

  \begin{definition}
    Let $C$ be any cube in $\ess_g$ and let $A$ be the set of elements in
    $\axis_g^+$ dual to the edges of $C$. Let $x,y$ be the two vertices of
    $C$ such that $A=[x,y]$. We will call $x$ the \emph{minimal vertex} of
    $C$ and $y$ the \emph{maximal vertex} of $C$. 

  \end{definition}

  \begin{lemma} \label{Lem:CubeAxis}
  
    Let $C \times \set{v}$ be a cube in $\mess_g \times \set{v} \subseteq
    \mess_g \times \fix_g$ with minimal and maximal vertices $x$ and $y$.
    Then $[x,y] \subseteq [x,gx]$. Thus, $y$ lies on some combinatorial
    axis of $g$ containing $x$. This axis lies inside $\mess_g \times
    \set{v}$. 

  \end{lemma}
  
  \begin{proof}
  
    Let $e = (x, z)$ be any oriented edge in $C \times \set{v}$ with
    initial vertex $x$, and let $H \in [x,y]$ be the half-space dual to
    $e$. Since $C \subset \mess_g$, the edge $e$ lies on a combinatorial
    axis of $g$. In particular, the vertex $z$ lies on a geodesic edge path
    from $x$ to $gx$. Since this geodesic can cross $\partial H$ only once,
    $gx \in H$. In other words, $H \in [x,gx]$. This is true for every $H
    \in [x,y]$, so $[x,y] \subseteq [x,gx]$. For the last conclusion, let
    $\alpha$ be the concatenation of geodesic edge paths from $x$ to $y$
    and from $y$ to $gx$. Then $\alpha$ does not cross any hyperplane
    twice, since such a hyperplane would separate $y$ from $x$ and $gx$.
    The concatenation of $\alpha$ and its $g$--translates is a
    combinatorial axis containing $x$ and $y$. The axis lies in
    $\mess_g\times\set{v}$ by $\gen{g}$--invariance of the product
    decomposition $\mess_g \times \fix_g$. \qedhere

  \end{proof}

  We now show that $\mess_g$ is always finite-dimensional.  
  
  \begin{lemma} \label{Lem:MFiniteDim}

    Suppose $g \in G$ is hyperbolic. Then 
    $\mess_g$ is finite-dimensional, with dimension bounded by the
    translation distance $\ell_g$ of $g$. 
  
  \end{lemma}
  
  \begin{proof}
    
    Recall that the distance $d(x,y)$ between two vertices is the same as
    the cardinality of $[x,y]$. Thus, for any $x \in \mess_g$, the
    cardinality of $[x,gx]$ is the same as the translation distance
    $\ell_g$. Let $C$ be any cube in $\mess_g$, let $v$ be any vertex
    of $\fix_g$, and let $x$ and $y$ be the minimal and maximal
    vertices of $C \times \set{v}$ in $\mess_g\times \fix_g$. The
    dimension of $C$ is the same as the cardinality of $[x,y]$. By
    \lemref{CubeAxis}, we always have $[x,y] \subseteq [x,gx]$, so the
    dimension of $C$ is bounded by $\ell_g$. This is true for all $C$
    in $\mess_g$, whence the result. \qedhere 
  
  \end{proof}
  
  Our goal now is to relate $\mess_g$ and $\ess_g$. It turns out that
  $\ess_g$ may have infinite dimension. An easy example showing that
  $\mess_g$ and $\ess_g$ can have different dimensions is the glide
  reflection in $\R^2$ defined by 
  \[ 
    g(x,y) = (y+1,x). 
  \] 
  Then $g$ has translation length $1$, so $\mess_g$ is
  $1$--dimensional by \lemref{MFiniteDim}, but $\ess_g = \R^2$. See
  \figref{Dimension}. 
  \begin{figure}[htp!]
  \begin{center}
    \includegraphics{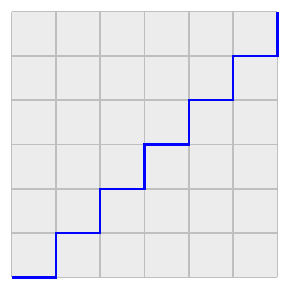}
  \end{center}
  \caption{A glide reflection $g$ with a unique combinatorial axis. We
    have $\dim(\mess_g) = 1$ and $\dim(\ess_g) = 2$. } 
  \label{Fig:Dimension}
  \end{figure}

  This example can be promoted to one in which $\ess_g$ is
  infinite-dimensional. Consider $\R^\Z$ with its standard integer cubing.
  Fix the origin $o=(0,0,\dotsc)$ and consider the subcomplex $X \subset
  \R^\Z$ generated by the vertices in $\R^\Z$ having at most finitely many
  non-zero coordinates. Then $X$ is an infinite-dimensional CAT(0) cube
  complex. Given $x \in X$, let $x_i$ denote its $i$--th coordinate. Let $g
  \from X \to X$ be defined by $g(x)_0 = x_{1} + 1$ and $g(x)_j = x_{j+1}$
  for all other $j$. Again, $g$ has translation length $1$, and $\mess_g$
  is $1$--dimensional, consisting of a single combinatorial axis with
  vertices $\set{g^n(o)}$. 

  Letting $H = [o,go]$, the set of half-spaces $\set{H, gH, \dotsc,
  g^{d-1}H}$ are pairwise transverse, and the $d$--dimensional cube $C$
  they cross in is contained in $\mess_{g^d}$. In particular, since
  $\ell_{g^d} = d$, we see that $\mess_{g^d}$ has dimension exactly $d$.
  Now $\ess_g$ is infinite-dimensional, since $\mess_{g^d} \subset
  \ess_{g^d} = \ess_g$ for all $d>0$. Note that in this example, $g$ has a
  combinatorial axis in $X$, but it has no CAT(0) axis; see \cite[Example
  II.8.28]{BH}. 
 
  The above discussion leads to the next definition.

  \begin{definition} \label{Def:NonTransverse}

    If $g \in G$ is hyperbolic, we say that $\gen{g}$
    acts \emph{non-transversely} on $\ess_g$ if, for every $H \in A_g$,
    $H$ and $gH$ are not transverse in $\ess_g$. Note that this occurs 
    if and only if $H$ and $gH$ are not transverse in $X$, by
    \lemref{Crossing}.

  \end{definition}
  
  Note that while the inclusion $\mess_g \into \ess_g$ induces a bijection
  on their sets of half-spaces, the partial orderings induced by
  inclusion on these two sets may be very different; equality holds if
  and only if $\mess_g = \ess_g$. In the following, we will establish
  some criteria for a square in $\ess_g$ to not be in $\mess_g$.
  
  \begin{lemma} \label{Lem:Transverse}

    Suppose $g \in G$ is hyperbolic. Then:
    \begin{enumerate}[label=\textup{(\arabic*)}]

      \item If $H$ and $gH$ are transverse in $\ess_g$, then any
        square in which they cross cannot be contained in $\mess_g$.
        \label{transverse1}

      \item Let $e$ and $e'$ be edges in $\mess_g$ with dual
        half-spaces $H, H' \in \axis_g^+$. If $e$ and $e'$ bound a
        square in $\ess_g$ which is not contained in $\mess_g$, then either
        $gH = H'$ or $gH' = H$. 
        \label{transverse2}

      \end{enumerate}
    
  \end{lemma}

  \begin{proof}
  
    Let $S$ be a square in $\ess_g$ in which $H$ and $gH$ cross. Let $x$ be
    the minimal vertex of $S$. If $S$ lies in $\mess_g$, then by
    \lemref{CubeAxis}, $H$ and $gH$ are in $[x,gx]$, which is a
    contradiction. This shows \ref{transverse1}. 

    For \ref{transverse2}, let $S$ be the square in $\ess_g$ bounded by $e$
    and $e'$. Let $x$ and $y$ be respectively the minimal and maximal
    vertices of $S$. Let $z$ be the vertex at which $e$ and $e'$ are
    incident. We claim that $z$ is distinct from $x$ and $y$. Since the
    union of $e$ and $e'$ contains three corners of $S$, exactly one corner
    of $S$, the one opposite $z$, is not contained in $\mess_g$. If
    $z=x$, then $y$ must be the corner not in $\mess_g$. But in this case,
    both half-spaces $H$ and $H'$ are contained in $[x,gx]$, by
    \lemref{EdgeAxis}, so $y = m(x,y,gx)$. This means that $y$ lies on
    a combinatorial axis for $g$, so $y$ is in $\mess_g$, a
    contradiction. Similarly, if $z=y$, then $x=m(g^{-1}y,x,y)$, so $x
    \in \mess_g$, again a contradiction. Therefore $x$, $y$, and $z$
    are distinct and lie in $\mess_g$. Now, without loss of
    generality, assume that $[x,z] = H$ and $[z,y]=H'$. Note that $z =
    m(x,z,gx)$ and $y = m(z,y,gz)$. Since $S$ is not contained in
    $\mess_g$, $H' \notin [x,gx]$. On the other hand, $H' \in
    [z,gz]$. Since $[z,gz] = [z,gx] \cup [gx,gz]$, we must have $H' =
    [gx,gz] = gH$. \qedhere 
    
  \end{proof}

  \begin{proposition} \label{Prop:Equality}

    Suppose $g \in G$ is hyperbolic. Then $\mess_g = \ess_g$ if and only
    if\/ $\gen{g}$ acts non-transversely on $\ess_g$.

  \end{proposition}

  \begin{proof}
    
    If $H \trans gH$ for some $H \in \axis_g^+$, then, by
    \lemref{Transverse}\ref{transverse1}, any square in $\ess_g$ in which
    they cross cannot lie in $\mess_g$. Therefore, $\mess_g = \ess_g$
    implies that $\gen{g}$ acts non-transversely on $\ess_g$.

    Now suppose $\gen{g}$ acts non-transversely on $\ess_g$. We claim that
    for any edge $e=(x,y)$ in $\ess_g$, if $x \in \mess_g$ then $y \in
    \mess_g$. To see this, let $H = [x,y]$. Replacing $g$ by $g^{-1}$ if
    necessary, we may assume that $H \in \axis_g^+$. If $y \notin \mess_g$,
    then $y \ne m(x,y,gx)$. In particular, $H \notin [x,gx]$. But $H$ must
    be contained in $[g^nx,g^{n+1}x]$ for some $n \in \Z$, and for this $n$
    we have $g^{-n} H \in [x,gx]$. Note that $n > 0$, since $x \not\in H$.
    Because $\gen{g}$ acts non-transversely, $g^{-n}H$ and $H$ cannot be
    transverse, and so $g^{-n} H \supset H$. Since $x \notin g^{-n}H$ and
    $H$ is the only half space separating $x$ and $y$, we must have $y
    \notin g^{-n} H$. But this contradicts the fact that $y \in H$. This
    finishes the proof of the claim. 

    To finish the argument, it suffices to observe that the $1$--skeleton
    of $\ess_g$ is connected, and therefore every vertex of $\ess_g$ is in
    $\mess_g$. \qedhere 

  \end{proof}
  
  \begin{remark} 

    If $\gen{g}$ acts non-transversely on $\ess_g$, then $\mess_{g^n} =
    \mess_g$ for all $n \ge 1$, because $\ess_{g^n} = \ess_g$. The latter
    equality holds because any axis for $g$ is an axis for $g^n$, and 
    therefore $\axis_{g^n} = \axis_g$.  

  \end{remark}

  \begin{corollary} \label{Cor:XFiniteDim}
   
   Suppose $g \in G$ is hyperbolic. The
   following statements are equivalent. 
    \begin{enumerate}[label=\textup{(\arabic*)}]
     \item There is an integer $k > 0$ such that $\gen{g^k}$
     acts non-transversely on $\ess_g$. \label{fdim1}
     \item $\ess_g = \mess_{g^k}$ for some $k > 0$. \label{fdim2}
     \item $\ess_g$ is finite-dimensional. \label{fdim3}
    \end{enumerate}

  \end{corollary}

  \begin{proof}
   
    First we show that \ref{fdim1} $\Longrightarrow$ \ref{fdim2}. Suppose
    that $g^k$ acts non-transversely on $\ess_g$. Since $\ess_g =
    \ess_{g^k}$, $g^k$ also acts non-transversely on $\ess_{g^k}$. By
    \propref{Equality}, $\mess_{g^k}=\ess_{g^k}$. 

    The implication \ref{fdim2} $\Longrightarrow$ \ref{fdim3} follows from
    \lemref{MFiniteDim}. 

    Now we show that \ref{fdim3} $\Longrightarrow$ \ref{fdim1}. Suppose
    $\ess_g$ has dimension $d$. For every $H \in \axis_g^+$, we claim
    that $H \supset g^n H$ for some $n$ satisfying $0 < n \le d$. Since
    $\ess_g$ has dimension $d$, the half-spaces 
    \[
      H, gH, g^2H, \dotsc, g^d H
    \]
    cannot all be pairwise transverse. Thus there 
    exist $i, j$ with $0 \leq i < j \leq d$ such that $g^i H \supset
    g^j H$, or equivalently, $H \supset g^{j-i} H$. Finally, taking
    $k=d!$, we have that $n \mid k$ whenever $0 < n
    \leq d$, and therefore $H \supset g^kH$ for all $H \in
    \axis_g^+$. Thus $\gen{g^k}$ acts non-transversely on
    $\ess_g$. \qedhere 

  \end{proof}
  
  Even though $\mess_g$ is not always a convex subspace of $X$, it is still
  always CAT(0), which we show next. The proof will make use of
  \lemref{Transverse}\ref{transverse2}. 
  
  \begin{proposition} \label{Prop:CAT0}
  
    Suppose $g \in G$ is hyperbolic. Then the cube complex $\mess_g$ is
    CAT(0). 

  \end{proposition}
   
  \begin{proof}
     
     Denote by $\half$ the set of half-spaces $\axis_g$ with partial order
     induced by $\mess_g$. That is: $H, H' \in \half$ are incomparable (or
     transverse) if and only if there is a square in $\mess_g$ in which
     they cross, and $H \geq H'$ if and only if $H \cap \mess_g \supseteq
     H' \cap \mess_g$. Apply the Sageev construction to $\half$ to obtain a
     CAT(0) cube complex $X(\half)$. There is a natural injective map $f
     \from \mess_g \to X(\half)$ defined by sending every vertex in
     $\mess_g$ to its associated principal ultrafilter on $\half$. This map
     identifies $\half$ with the set of half-spaces of $X(\half)$. Let $Y$
     be the image of $\mess_g$. We now proceed to show that $Y= X(\half)$,
     which will imply that $\mess_g$ is CAT(0). 
 
     We claim that for any edge $e=(y,y')$ in $X(\half)$, if $y \in Y$ then
     $y' \in Y$. The result follows, since the $1$--skeleton of $X(\half)$
     is connected. To prove the claim, let $H \in \half$ be the half-space
     dual to $e$. Replacing $g$ by $g^{-1}$ if necessary, we may assume
     that $H \in \half^+$. Let $x \in \mess_g$ be such that $f(x) = y$.
     Since $H$ must appear in every axis of $g$ passing through $x$ and $x
     \notin H$, there exists a geodesic path $x=x_0,\ldots,x_{n+1}=x'$ in
     $\mess_g$ such that $H = [x_n,x_{n+1}]$. If $n=0$, then $f(x_1) = y'$
     and we are done. Now suppose that $n > 0$. Let $H_i = [x_i,x_{i+1}]$
     for each $i$. If $H_i \supset H$ for some $i<n$, then since $x \notin
     H_i$, $f(x)=y \notin H_i$. But $y' \in H$, and hence $y' \in H_i$, but
     this is impossible as $y$ and $y'$ are separated by exactly one
     half-space $H$. Thus $H_i \trans H$ for $i=0,\ldots,n-1$. We claim now
     that for each $i$, there is a square $S_i$ in which $H_i$ and $H$
     cross, and $x_i$ is the minimal vertex of $S_i$. Let $e_i =
     (x_i,x_{i+1})$ for each $i$. By assumption, $e_{n-1}$ and $e_n$ are
     dual to two transverse half-space $H_{n-1}$ and $H_n$. If they do not
     bound a square in $\mess_g$, then by \lemref{Transverse}, $gH_{n-1} =
     H_n$ and there is no square in $\mess_g$ in which $H_{n-1}$ and $H_n$
     can cross. This contradicts that $H_{n-1}$ and $H_n$ are transverse in
     $\mess_g$, so $e_{n-1}$ and $e_n$ must bound a square $S_{n-1}$ in
     $\mess_g$. By the same reasoning, the edge $e'$ parallel to $e_n$ in
     $S_{n-1}$ and $e_{n-2}$ bound $S_{n-2}$, in which $x_{n-2}$ is the
     minimal vertex. Repeating in this way, we find the square $S_0$, with
     minimal vertex $x_0$ and in which $H_0$ and $H$ cross. In $S_0$ there
     is an edge $(x_0,v)$ dual to $H$, and $f(v) = y'$. \qedhere 
    
  \end{proof}

  The following proposition will be used in the next section.

  \begin{proposition} \label{Prop:Descending}

      Suppose $g \in G$ is hyperbolic, and that
      $\gen{g}$ acts non-transversely on $\ess_g$. Let $C$
      be a cube in $\ess_g$ of maximal dimension, $A$ the set of
      half-spaces in $\axis_g^+$ dual to $C$, and $o$ the
      minimal vertex of\/ $C$. Then the following statements hold. 

      \begin{enumerate}[label=\textup{(\arabic*)}]

        \item For every pair of half-spaces $H, H' \in A$, either $H \trans
        gH'$ or $H \supset gH'$. \label{desc1}

        \item $K \in [o,go]$ if and only if there exist $H, H' \in A$ such
        that $H \supseteq K \supset g H'$. \label{desc2}

        \item For every $K \in \axis_g^+$, there exist $r,s \in \Z$ and $H,
        H' \in A$ such that \label{desc3} $g^r H \supset K \supset g^s
        H'$.  

      \end{enumerate}
    
  \end{proposition}

  \begin{proof}

     Since $\ess_g = \mess_g$ (by \propref{Equality}), there is an
     axis $L$ for $g$ containing $o$. It follows that $o \not\in gH'$, 
     because $o \not\in H'$ (the unique edge $e$ in $L$ dual to $H'$
     separates $o$ from $ge$). Let $o^+$ be the maximal vertex of
     $C$. By \lemref{CubeAxis}, $o^+$ is on a geodesic from $o$ to $g
     o$. Suppose $H$ and $gH'$ are not transverse. Then $o^+ \notin
     gH'$, because all half-spaces in $[o, o^+]$ are transverse. Now
     $o^+ \in H - gH'$, showing that $H \not\subset gH'$. Thus $H
     \supset gH'$ and \ref{desc1} holds. 

     For statement \ref{desc2}, note that $A \subseteq [o,g o]$,
     again by \lemref{CubeAxis}. If $H \supseteq K \supset g H'$ for  
     some $H, H' \in A$, then $o \notin K$ and $g o \in K$, and
     therefore $K \in [o,g o]$. On the other hand, both $C$ and $g C$
     have maximal dimension, so for every $K \in [o,g o]$, there exist
     $H, H' \in A$ such that $H$ and $gH'$ are comparable (or equal)
     to $K$. Because $A \subset [o,go]$ and $gA \cap [o,go] =
     \emptyset$, we must have $H \supseteq K \supset g H'$. 
     
     Finally, for \ref{desc3}, we observe that 
     \[ 
       \axis_g^+ = \bigcup_{n \in \Z} [g^{n}o,g^{n+1}o].
     \] 
     Suppose $K \in [g^{n}o,g^{n+1}o]$. Applying \ref{desc2} to
     $g^{-n}K$, there exist $H, H' \in A$ such that $g^{n} H \supseteq
     K \supset g^{n+1} H'$. Since $\gen{g}$ acts non-transversely, we
     also have $g^{n-1} H \supset g^{n} H$. The conclusion follows.
     \qedhere

  \end{proof}

\section{Non-transverse actions and efficient quasimorphisms} 

  \label{Sec:Non-transverse}

  Here we give a general construction of a large family of quasimorphisms
  on groups acting on CAT(0) cube complexes. For the construction to
  succeed (i.e.\ to achieve bounded defect) we require one assumption. 

  \begin{definition} \label{Def:NonTransverse2}
    Let $X$ be a CAT(0) cube complex with an action by $G$. The action is
    \emph{non-transverse} if it is without inversion and also satisfies: 
    there do not exist $H \in \half(X)$, $g \in G$ with $H \trans gH$.
    
    This definition agrees with the earlier \defref{NonTransverse} in
    the case of $\gen{g}$ acting on $\ess_g$. First, such an action is
    always without inversion. Also, if $H \in \axis_g$ and $H$ and
    $gH$ are not transverse, then $H$ and $gH$ are nested by
    \remref{TransNest}; hence $H$ and $g^k H$ are not transverse for
    any $k$. 
    
  \end{definition}
 
  Let $X$ be a CAT(0) cube complex with a non-transverse action by $G$. Let
  $\g$ be a segment in $X$, and consider the set $G \gamma = \set{ g
    \gamma \st g \in G}$; elements of this set are called \emph{copies}
  of $\gamma$. Define the map $c_\g \from X^2 \to \Z$ which assigns to
  each pair $(x,y)$ the maximal cardinality of a pairwise non-overlapping
  collection of copies of $\gamma$ in $[x,y]$. 
  
  Define 
  \begin{equation} \label{omegadef}
    \omega_\g(x,y) = c_\g(x,y)-c_{\bg}(x,y).
  \end{equation}
  Observe that $\omega_\g(y,x) = -\omega_\g(x,y) $ and $\omega_\g(gx,gy)
  = \omega_\g(x,y)$ for all $g \in G$. 

  \begin{lemma} \label{Lem:Juncture}

    If the action is non-transverse, then for all $x,y,z \in X$ with $y =
    m(x,y,z)$, there is a bound 
    \[ 
      \abs{ \omega_\g(x,z) - \omega_\g(x,y) - \omega_\g(y,z)} \le 2.
    \] 

  \end{lemma}
  
  \begin{proof}
    
    By definition, 
    \[ 
      \abs{ \omega_\g(x,z) - \omega_\g(x,y) - \omega_\g(y,z)} = \abs{
      \Big( c_\g(x,z) - c_\g(x,y) - c_\g(y,z) \Big) - \Big( c_{\bg}(x,z) -
      c_{\bg}(x,y) - c_{\bg}(y,z) \Big) }.
    \] 
    It will suffice to show that 
    \begin{gather}
      c_\g(x, z) \leq c_\g(x,y) + c_\g(y,z) + 1 \label{eqn1}\\
      \intertext{and} 
      c_\g(x, y) + c_\g(y, z) \leq c_\g(x,z) + 1, \label{eqn2}
    \end{gather}
    together with analogous statements for $\bg$. 

    Let $\set{g_1 \g, \dotsc, g_n \g}$ be a collection
    of non-overlapping copies of $\g$ in $[x,z]$ of cardinality $n =
    c_\g(x,z)$. By \lemref{Non-Over} these copies are pairwise
    nested, and hence up to re-indexing we can assume that 
    \begin{equation}\label{eqn3}
      g_1 \g > \dotsm > g_n \g.
    \end{equation}
    If $g_k\g \not\in [x,y] \cup [y,z]$ for some $k$, then $y$ 
    separates two half-spaces in $g_k \g$. It follows from
    \eqref{eqn3} that $g_i \g \subseteq [x,y]$ for every $i < k$ and $g_i
    \g \subseteq [y,z]$ for every $i > k$. Thus $c_\g(x,y) + c_\g (y ,z)
    \geq n - 1$, proving \eqref{eqn1}. 

    Now let $k = c_\g(x,y)$ and $\ell = c_\g(y,z)$. Let $A = \set{g_1\g,
      \dotsc, g_k\g}$ be a non-overlapping collection of copies of $\g$
    in $[x,y]$ and $B = \set{g_{k+1}\g, \dotsc, g_{k+\ell}\g}$ a
    non-overlapping collection of copies in $[y,z]$. As above, by
    \lemref{Non-Over}, we may re-index $A$ and $B$ to arrange that
    \[
      g_1\g > \dotsm > g_k\g \ \ \text{ and } \ \ g_{k+1}\g > \dotsm >
      g_{k+\ell}\g. 
    \]
    We claim that $g_i\g$ and $g_j\g$ (with $i < j$) cannot overlap unless
    $i=k$ and $j=k+1$. Discarding $g_k\g$, then obtain a
    non-overlapping collection in $[x,z]$ of cardinality $k + \ell -
    1$, proving \eqref{eqn2}. 

    To prove the claim, suppose that $g_i\g \in A$ and $g_j\g \in B$
    overlap. Then there are half-spaces $H, H' \in \g$ such that $g_iH
    \trans g_jH'$ in $X$. If $i < k$ then $g_k H' \trans g_jH'$, because
    $g_iH \supset g_k H'$ and $y \in g_kH' - g_j H'$. However, this
    contradicts the assumption of a non-transverse action. Hence $i =
    k$. Similarly, if $j > k+1$ then $g_i H \trans g_{k+1} H$ because
    $g_{k+1}H \supset g_j H'$ and $y \in g_i H - g_{k+1} H$. Again,
    this contradicts non-transversality of the action, and therefore $j =
    k+1$. This proves the claim, and equation \eqref{eqn2}. Finally, note
    that the analogues of \eqref{eqn1} and \eqref{eqn2} for $\bg$ are
    entirely similar. \qedhere 
  
  \end{proof}

  Next define $\delta \omega_\g(x,y,z) = \omega_\g(x,y) + \omega_\g(y,z)
  + \omega_\g(z,x)$. 
  
  \begin{lemma} \label{Lem:Coboundary} 

    If the action is non-transverse, then for all $x,y,z \in X$ there is
    a bound $\abs{\delta \omega_\g(x,y,z)} \leq 6$. 
    
  \end{lemma}

  \begin{proof}
    
    Let $m= m(x,y,z)$. By the previous lemma, $\abs{ \omega_\g(a, b) -
    \omega_\g(a,m) - \omega_\g(m, b)}\leq 2$, where $a,b \in \set{x,y,z}$
    are distinct. Then 
    \begin{align*}
    \abs{ \delta \omega_\g(x, y,z) }
    & =  \left\lvert \omega_\g(x, y) + \omega_\g(y, z) + \omega_\g(z,x) \right. \\
    & \quad + \left. \omega_\g(x, m) - \omega_\g(x,m) + \omega_\g(y, m)
    - \omega_\g(y, m) + \omega_\g(z,
    m)- \omega_\g(z, m) \right\rvert  \\
    & \leq  \abs{ \omega_\g(x, y) - \omega_\g(x,m)- \omega_\g(m,y) }
    + \abs{ \omega_\g(y, z) - \omega_\g(y,m)- \omega_\g(m,z)} \\
    & \quad + \abs{ \omega_\g(z, x) -\omega_\g(z,m)- \omega_\g(m,x) } \\ 
    & \leq 6. \qedhere
    \end{align*}

  \end{proof}

  At this point we are ready to define quasimorphisms associated to
  $\g$. We will define two functions, $\psi_\g$ and $\varphi_\g$, 
  which produce the \emph{same} homogeneous quasimorphism
  $\widehat{\psi}_\g = \widehat{\varphi}_\g$. The second function
  $\varphi_\g$ has the definition we want to use, but $\psi_\g$ is needed
  to establish the bound on defect. 

  Fix a base vertex $\O \in X$ and define $\psi_\g \from G
  \to \R$ by 
  \begin{equation} \label{psidef}
    \psi_\g(g) = \omega_\g( \O, g\O). 
  \end{equation}

  Next, for each $g\in G$ choose a vertex $x_g \in X_g$. Define 
  $\varphi_\g \from G \to \R$ by 
  \begin{equation} \label{phidef}
    \varphi_\g(g) = \omega_\g(x_g, gx_g).
  \end{equation}

  \begin{lemma}

    If the action is non-transverse, then $\psi_\g$ is a quasimorphism of
    defect at most\/ $6$. 

  \end{lemma}
  
  \begin{proof}

    For any $g_1, g_2 \in G$ we have 
    \begin{align*}
    \abs{ \psi_\g(g_1g_2) - \psi_\g(g_1)-\psi_\g(g_2)}
    & = \abs{ \omega_\g(\O, g_1g_2 \O) - \omega_\g(\O ,g_1 \O)-
    \omega_\g(\O, g_2\O)}\\
    &= \abs{\omega_\g(\O, g_1g_2 \O) + \omega_\g(g_1\O, \O) + \omega_\g(g_2
    \O, \O)}\\
    &= \abs{ \omega_\g(\O, g_1g_2\O) + \omega_\g(g_1\O, \O)+
    \omega_\g(g_1g_2\O,g_1\O) }\\
    &= \abs{ \delta \omega_\g(\O, g_1g_2\O, g_1\O)} \\
    & \leq 6,
    \end{align*} 
    by \lemref{Coboundary}. \qedhere

  \end{proof}

  \begin{lemma} \label{Lem:Defect} 
   
    If the action is non-transverse, then $\psi_\g-\varphi_\g$ is
    uniformly bounded. Hence $\varphi_\g$ is a quasimorphism,
    $\widehat{\varphi}_\g = \widehat{\psi}_\g$, and
    $\widehat{\varphi}_\g$ has defect at most\/ $12$. 

  \end{lemma}

  \begin{proof}

    For any $g \in G$ we have 
    \begin{align*}
    \abs{ \psi_\g(g) - \varphi_\g(g)} 
    & = \abs{ \omega_\g(\O, g\O) - \omega_\g(x_g, gx_g) 
        + \omega_\g(g\O,x_g) - \omega_\g(g\O,x_g)} \\
    & = \left\lvert \omega_\g(\O, g\O) + \omega_\g(g\O,x_g) -
    \( \omega_\g(g\O,x_g) + \omega_\g(x_g, gx_g) \) \right. \\ 
    & \quad + \left. \omega_\g(x_g,\O) - \omega_\g(x_g,\O) \right\rvert \\
    & \leq \abs{ \omega_\g(\O, g\O) + \omega_\g(g\O,x_g) + \omega_\g(x_g,\O)} \\
    & \quad + \abs{ \omega_\g(g\O,x_g) + \omega_\g(x_g, gx_g) +
      \omega_\g(x_g,\O) } \\  
    & = \abs{ \delta \omega_\g(\O, g\O, x_g)} + \abs{ \delta
    \omega_\g(g\O, x_g, gx_g)} \\ &\leq 12, 
    \end{align*}
    by \lemref{Coboundary}. This shows that $\psi_\g - \varphi_\g$ is
    uniformly bounded. The other conclusions follow immediately from
    \lemref{Bounded} and \lemref{Homogenization}. \qedhere

  \end{proof}

  Note that that the equality $\widehat{\varphi}_\g =
  \widehat{\psi}_\g$ also implies that this quasimorphism is
  independent of the choices of basepoints used to define
  $\varphi_{\g}$ and $\psi_{\g}$. 
  
  \begin{remark} \label{Rem:CFL6.6} 

    The bounds in the preceding lemmas can be improved by a factor of $2$
    in the special case where $X$ is a $1$--dimensional CAT(0) cube
    complex (that is, a simplicial tree). In this case, half-spaces are
    never transverse, so two segments overlap if and only if they have
    non-empty intersection. We obtain an improvement in equation
    \eqref{eqn2}, which becomes instead 
    \begin{equation}
      c_\g(x,y) + c_\g(y,z) \leq c_\g(x,z) \tag{\ref{eqn2}'}
    \end{equation}
    since there is no need to discard $g_k\g$ from the collection of
    segments in $[x,z]$. This leads to the bounds 
    \[
      \abs{ \omega_\g(x,z) - \omega_\g(x,y) - \omega_\g(y,z)} \le 1
    \]
    in \lemref{Juncture}, $\abs{\delta \omega_\g(x, y, z)} \leq 3$ in
    \lemref{Coboundary}, and a defect of at most $6$ in
    \lemref{Defect}. Thus we have a new proof of Theorem 6.6 of
    \cite{CFL}, which is the statement that these quasimorphisms have
    defect at most $6$. 

    At this point, one could enhance \thmref{Main} to say that
    $\scl(g) \geq 1/12$ when $X$ is a tree, but this already follows from
    Theorem 6.9 of \cite{CFL}. 

  \end{remark}

  \subsection*{Bounded cohomology of right-angled Artin groups}

  Recall that for any group $G$, we denote by $\widetilde{\QH}(G)$ the
  space of homogeneous quasimorphisms on $G$, modulo homomorphisms. It
  is a subspace of the second bounded cohomology $H^2_b(G;\R)$. 

  If $G$ is a non-abelian
  right-angled Artin group, then there is a retraction onto a
  non-abelian free subgroup $H$, which induces injections $H^2_b(H;\R)
  \into H^2_b(G;\R)$ and $\widetilde{\QH}(H) \into
  \widetilde{\QH}(G)$. Therefore the spaces $H^2_b(G;\R)$ and
  $\widetilde{\QH}(G)$ are infinite-dimensional (we thank the referee
  for suggesting this viewpoint). The next result provides
  another way of seeing this, using the quasimorphisms
  $\widehat{\varphi}_{\gamma}$. We note that the quasimorphisms
  obtained below appear to be quite different from those arising
  via the retraction to $H$, even though their restrictions to $H$
  agree. 

  \begin{proposition}\label{Prop:linear-indep}
  
    Let $G = A_{\Gamma}$ be a non-abelian right-angled Artin group,
    and $X$ the natural cube complex on which $G$ acts. Then there is
    an infinite family $\set{\gamma_i}$ of segments in $\half(X)$ such
    that the homogeneous quasimorphisms
    $\set{\widehat{\varphi}_{\gamma_i}}$ are linearly independent in 
    $\widetilde{\QH}(G)$. 

  \end{proposition}

  \begin{proof}

    Let $a,b$ be standard generators of $G$ which generate a free
    subgroup $H < G$. We shall show that every ``non-overlapping''
    Brooks quasimorphism on $H$ is the restriction of a quasimorphism
    $\widehat{\varphi}_{\g}$ for some $\g$. By 
    \cite[Proposition 5.1]{Mitsumatsu} there is an infinite linearly
    independent family of Brooks quasimorphisms in
    $\widetilde{\QH}(H)$, and their extensions will be independent in
    $\widetilde{\QH}(G)$. 

    If $w$ is a reduced word in $a,b$, the non-overlapping
    Brooks quasimorphism $\widehat{B}_w \from \langle a, b\rangle \to
    \R$ is the homogenization of the quasimorphism $B_w = C_w -
    C_{\overline{w}}$, where $C_{w}(g)$ is the maximal number of
    disjoint subwords of $g$ (considered as a reduced word) which
    equal $w$. In the $1$--skeleton of $X$ there is an edge path
    labeled by the word $w$, 
    starting at a vertex $x$ and ending at $y$. Because $a$ and $b$ do
    not commute, no two half-spaces dual to this segment can
    cross. Thus $[x,y]$ is a \emph{segment}, which we denote by
    $\g(w)$. Modulo the $G$--action on $X$, $\g(w)$ is uniquely
    determined by $w$. 

    We claim that $B_w(g) = \varphi_{\g(w)}(g)$ for every $g \in H$, and
    therefore $\widehat{B}_w$ is the restriction of
    $\widehat{\varphi}_{\g(w)}$ to $H$. If an element $g \in H$ is
    considered as a reduced word, it has a combinatorial axis in $X$
    which is labeled by $g^{\infty}$. The half-spaces dual to this
    axis never cross, and so the partial ordering on $\axis^+_g$ is a
    linear ordering. Thus $\ess_g$ is one-dimensional and the axis is
    an embedded copy of $\ess_g$ in $X_g \subset X$. Let $x_g$ be
    a vertex on this axis at the beginning of the word $g$; this is
    the basepoint for the definition of $\varphi_{\g(w)}(g)$. Now segments
    in $[x_g, g x_g]$ correspond bijectively with subwords of $g$
    via the labelling, and so $B_w(g) = \varphi_{\g(w)}(g)$. \qedhere

  \end{proof}

\section{Dilworth's theorem and equivariant embeddings}
  
  \label{Sec:Dilworth}

  Let $P$ be a partially ordered set. A \emph{chain} in $P$ is a subset
  that is linearly ordered. A chain is \emph{maximal} if it is not
  properly contained in another chain. An \emph{antichain} in $P$ is a
  subset such that no two elements are comparable to each other. The
  \emph{width} of $P$ is the maximal cardinality of an antichain (which
  may be $\infty$). 
  
  \begin{lemma}[Dilworth's theorem] \label{Lem:Dilworth}
    
    Let $P$ be a partially ordered set. If $P$ has width $d < \infty$ then
    there is a partition of $P$ into $d$ chains. Furthermore, there is such
    a partition such that one of the chains is maximal. 

  \end{lemma}

  This first conclusion is the traditional statement of the theorem. 
  The second claim can be proved using Hausdorff's maximal principle. 

  The partition of $P$ into chains provided by the theorem will be
  called a \emph{Dilworth partition}. 
  
  \begin{definition}

    Let $P$ be a partially ordered set that admits an order-preserving
    free action by an 
    infinite cyclic group $\gen{g}$. Let $A$ be an antichain in $P$. We say
    $A$ is \emph{$\gen{g}$--descending} if $ga \not> a'$ for all $a, a'\in
    A$. We say that $A$ \emph{spans} $P$ if for each $p \in P$ there exist
    $a, a' \in A$ and $r, s \in \Z$ such that $g^r a > p > g^s a'$. 

  \end{definition}

  We further define the subsets
    \begin{align*}
      [A,gA] &= \set{p \in P \st x \geq p \geq y
      \text{ for some } x,y \in (A \cup gA)}\\
      &= \set{p \in P \st x \geq p \geq y \text{ for some } x\in A,  y
      \in gA }, \text{ if $A$ is $\gen{g}$--descending} 
    \end{align*} and 
    \[ 
      [A, gA) \ = \ [A, gA] - gA, \quad (A, gA] \ = \ [A, gA] - A.
    \]
  
  \begin{lemma}[Equivariant Dilworth theorem] \label{Lem:EquiDilworth}
    
    Let $P$ be a partially ordered set of width $d < \infty$ with an
    order-preserving free
    action by an infinite cyclic group $\gen{g}$. Suppose further that
    there is an antichain $A$ of cardinality $d$ that is is both
    $\gen{g}$--descending and spans $P$. Then there is a
    $\gen{g}$--invariant partition of $P$ into $d$ chains whose
    intersection with $[A, gA]$ is a Dilworth partition which includes a
    maximal chain in $[A, gA]$. 

  \end{lemma}
  
  \begin{proof}

    Apply Lemma \ref{Lem:Dilworth} to the partially ordered set $[A, gA]$  
    to obtain a partition by chains $[A, gA] = Q_1 \cup \cdots \cup
    Q_d$, with $Q_1$ maximal in $[A,gA]$. Each $Q_i$ contains exactly one 
    element of $A$ and one of $gA$, since these are antichains of
    cardinality $d$. We claim that these are the maximal
    and minimal elements, respectively, of $Q_i$. 
    
    Suppose the unique element $a$ of $A \cap Q_i$ is not maximal in
    $Q_i$. If $p\in Q_i$ satisfies $p > a$ then, since $p \in (A, gA]$,
    we must have $x > p$ for some $x \in A$. Then $x > p > a$, 
    contradicting that $A$ is an antichain. By a similar argument, the
    unique element of $gA \cap Q_i$ is minimal in $Q_i$. 

    Now label the elements of $A$ and define a permutation $\sigma$ as
    follows: $a_i$ is the maximal element of $Q_i$ and $g a_{\sigma(i)}$
    is the minimal element of $Q_i$, for $i = 1, \ldots, d$. Define the
    sets 
    \[ 
      P_i = \bigcup_{k \in \Z} g^k Q_{\sigma^k(i)}
    \]
    for each $i$. Note that for each $k$, the element $g^k a_{\sigma^k(i)}$
    is both the minimum of $g^{k-1}Q_{\sigma^{k-1}(i)}$ and the maximum of
    $g^k Q_{\sigma^k(i)}$. Hence $P_i$ is a chain, being a concatenation of
    chains. Since $\gen{g}$ acts freely on $P$, the chains $P_i$ are
    disjoint. Their union is the set $\bigcup_{k \in \Z} g^k [A, gA]$. It
    is immediate that $g P_{\sigma(i)} = P_i$, so the partition of
    $\bigcup_{k \in \Z} g^k [A, gA]$ by the chains $P_i$ is preserved by
    $g$. It remains to show that this set is all of $P$. 

    Given $p \in P$, let $a, a', r, s$ be given such that $g^r a > p >
    g^s a'$. First we claim that $s > r$. If not, then $r > s$. Writing
    $a = a_i$ we have $g^r a_i < g^{r-1} a_{\sigma^{-1}(i)} < \cdots <
    g^s a_{\sigma^{s-r}(i)}$, whence $g^s a' < g^s a_{\sigma^{s-r}(i)}$,
    a contradiction since $g^s A$ is an antichain. 

    Next we show that $p \in \bigcup_{k\in \Z} g^k [A,gA]$, by induction on
    $s-r$. Clearly we may assume that $p \not\in \bigcup_{k\in\Z} g^kA$. If
    $s-r = 1$ then we already have $p \in g^r [A,gA]$. If $s - r > 1$ then
    consider the (maximal) antichain $g^{r+1}A$. It contains an element
    $g^{r+1}a''$ which is comparable to $p$, by maximality. Then either
    $g^r a > p > g^{r+1}a''$ or $g^{r+1}a'' > p > g^s a'$, and in either
    case the induction hypothesis yields the conclusion that $p \in
    \bigcup_{k \in \Z} g^k [A, gA]$. \qedhere

  \end{proof}
  
  \subsection*{Equivariant Euclidean embeddings}

  Let $\R^d$ be equipped with its standard integer cubing. Given a
  coordinate $i$ and an integer $n$, we define: 
  \[ 
    H^i_n = \set{(x_1, \dotsc, x_d) \in \R^d \st x_i \geq n+ 1/2}. 
  \]
  Note that $H^i_n$ and $H^j_m$ are transverse in $\R^d$ if and only if $i
  \not= j$. We also define $\half^i = \set{ H^i_n \st n
    \in \Z}$, and set 
  \[
  \half^+(\R^d) = \half^1 \sqcup \dotsb \sqcup \half^d. 
  \] 
  The set of half-spaces of $\R^d$ is $\half(\R^d) = \half^+(\R^d) \sqcup
  \half^-(\R^d)$, where $\half^-(\R^d) = \set{ \bH \st H \in
  \half^+(\R^d)}$. 
  
  \begin{proposition} \label{Prop:TautEmbedding}

    Let $g \in G$ be hyperbolic and suppose $\gen{g}$ acts
    non-transversely on $\ess_g$. Let $C$ be a cube in $\ess_g$ of 
    dimension $d = \dim(\ess_g)$ and let $A$ be the set of elements of
    $\axis_g^+$ dual to the edges of\/ $C$. Then 
    there exist a $\gen{g}$--action on $\R^d$ and a $\gen{g}$--equivariant
    isometric embedding $\phi \from \ess_g \into \R^d$ satisfying the
    following properties: 
    \begin{enumerate}[label=\textup{(\arabic*)}]
      \item $\phi(C) = [0,1]^d \subset \R^d$. \label{taut1}
      \item The induced map $\phi_* \from \axis_g \to \half(\R^d)$ is a
      bijection, with $\phi_*(\axis_g^+) = \half^+(\R^d)$. \label{taut2}
      \item The set $[A,gA] \cap \phi_*^{-1}(\half^1)$ is tightly nested in
      $\ess_g$. \label{taut3}
      \item $[A, g A) = [o, g o]$, where $o$ is the minimal vertex
        of\/ $C$. \label{taut4}
   \end{enumerate}

  \end{proposition}
  
  By property \ref{taut2}, we can henceforth identify elements of
  $\axis_g^+$ with their corresponding half-spaces in $\half^+(\R^d)$ and
  we shall denote the corresponding decomposition as $\axis_g^+ =
  \half^1\sqcup \cdots \sqcup \half^d$. By property \ref{taut3}, every
  subsegment of $[A,gA] \cap \half^1$ is tightly nested in $\axis_g^+$. We
  will call $\gamma = [A,gA) \cap \half^1$ the \emph{taut segment} of
  the embedding; $[A,gA] \cap \half^1$ the \emph{extended} taut
  segment; and the map $\phi$ a taut $\gen{g}$--equivariant embedding
  of $\ess_g$ into $\R^d$. 
  
  \begin{proof}[Proof of \propref{TautEmbedding}]

    Let $P=\axis_g^+$ be partially ordered by inclusion; note that
    this partial ordering is preserved by $\gen{g}$. It has width $d$
    since $\ess_g$ has dimension $d$, and $A$ is an antichain of
    cardinality $d$. By \propref{Descending}\ref{desc1}, $A$ is 
    $\gen{g}$--descending. By \propref{Descending}\ref{desc3}, $A$
    spans $P$. We also have that $[A,g A) = [o,g o]$, by
    \propref{Descending}\ref{desc2}, and therefore \ref{taut4} holds. 

    Now apply \lemref{EquiDilworth} to $P$ to obtain a
    $\gen{g}$--invariant partition of $P$ into $d$ chains $P_1,
    \dotsc, P_d$. Without loss of generality, we may assume that $P_1
    \cap [A,gA]$ is a maximal chain in $[A,gA]$. Note that each chain
    $P_i$ is bi-infinite, since $\gen{g}$ acts freely on it. 

    For each $i$, let $K_i$ be the unique element of $P_i \cap
    A$. There is an order-preserving bijection $P_i \to \half^i$ 
    induced by sending $K_i$ to $H^i_0$. The resulting bijection
    $\axis_g^+ \to \half^+(\R^d)$ extends to a bijection $\phi_* \from
    \axis_g \to \half(\R^d)$ in an obvious way. 
  
    We now define an isometric embedding $\phi \from \ess_g \into \R^d$
    whose induced map on half-spaces is $\phi_*$. For any $x \in \R^d$,
    denote by $x_i$ its $i$--th coordinate. Let $v \in \ess_g$ be any
    vertex. For each $i$, let $K \in P_i$ be the largest element such that
    $v \notin K$. Define $\phi(x)_i = n$, where $\phi_*(K)=H^i_n$. This
    defines an embedding of the vertices of $\ess_g$ into $\R^d$. Two
    vertices $v$ and $w$ in $\ess_g$ bound an oriented edge $(v,w)$ dual
    to $K \in P_i$ if and only if $\phi(w)_i = \phi(v)_i + 1$ and
    $\phi(w)_j = \phi(v)_j$ for all $j \neq i$. Therefore $\phi$ extends
    to an embedding of the $1$--skeleton of $\ess_g$, and hence
    extends to all of $\ess_g$. It is immediate that $\phi$ induces
    the same map on half-spaces as $\phi_*$, so property \ref{taut2}
    holds. By construction, $o$ is mapped to the origin and the
    vertex of $C$ opposite $o$ is mapped to $(1, \dotsc, 1)$, so
    \ref{taut1} holds. 
    
    By $\gen{g}$--invariance of the partition, there is a permutation
    $\sigma$ such that $gP_{\sigma(i)} = P_i$. For each $i$ let $n_i =
    \phi(g(o))_i$. That is, $n_i$ is the shift given by the bijection
    $g \from P_{\sigma(i)} \to P_i$, relative to the basepoints
    $K_{\sigma(i)}$ and $K_i$. Then, for every vertex $v \in \ess_g$,
    we have 
    \[ 
      \phi(g(v))_i = \phi(v)_{\sigma(i)} + n_i. 
    \] 
    This allows us to define an action of $\gen{g}$ on $\R^d$: for
    every $x \in \R^d$ let $g(x)_i = x_{\sigma(i)} + n_i$. By
    construction, $\phi$ is $\gen{g}$--equivariant. 

    For property \ref{taut3}, note that $[A, gA] \cap
    \phi_*^{-1}(\half^1) = P_1 \cap [A,gA]$. Suppose $K' \supset K
    \supset K''$ for some $K \in \axis_g^+$ and $K', K'' \in P_1 \cap
    [A, gA]$. There is a unique $i \in \Z$ such that $K \in
    [g^iA,g^{i+1}A)$. If $i < 0$, then $K \supset H$ for some $H\in
    A$, which contradicts $K' \supset K$. If $i > 0$, then $gH 
    \supset K$ for some $H \in A$. But $gH \supset K$ contradicts $K
    \supset K''$, so $K \in [A,gA]$. By maximality, $K \in P_1 \cap
    [A,gA]$. This shows that $P_1 \cap [A,gA]$ is tightly
    nested. \qedhere 
   
  \end{proof}

  \subsection*{An example}
  
  \label{example}
  
  Let $A_{\Gamma}$ be the right-angled Artin group with $\Gamma$ the
  pentagon graph: 
  \[
    A_{\Gamma} \ = \ \gen{ a, b, c, d, e \mid [a,b] = [b,c] = [c,d] =
      [d,e] = [e,a] = 1}. 
  \]
  The element $g = abcde$ is hyperbolic, and part of its essential 
  characteristic set $\ess_g$ is shown in \figref{Pure}. The
  figure also demonstrates the equivariant embedding $\ess_g \into
  \R^2$. The action of $g$ on $\R^2$ (extending the natural action on
  $\ess_g$) is by a glide reflection whose axis is a diagonal line
  through the center of the figure. The $A_{\Gamma}$--invariant
  labeling of the edges of $\ess_g$ by generators of $A_{\Gamma}$ is
  also shown. 
  \begin{figure}[htp!]
  \begin{center}
    \includegraphics{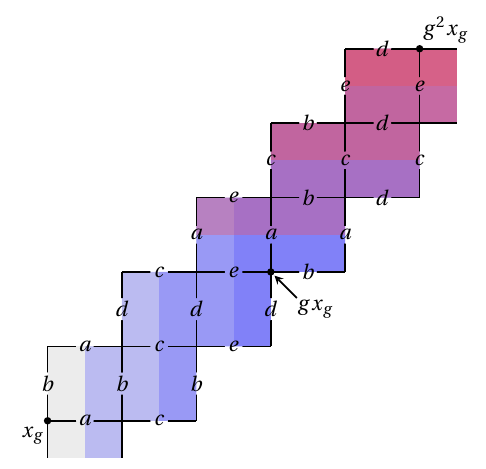}
  \end{center}
  \caption{The subcomplex $\ess_g$ embedded in $\R^2$ for the element
  $g = abcde$ in the pentagon RAAG. The extended action of $g$ on
  $\R^2$ is by a glide reflection. The three blue half-spaces (with
  labels $a$, $c$, $e$) are taken to the three red half-spaces. }  
  \label{Fig:Pure}
  \end{figure}
  For this particular choice of $g$, the essential characteristic set
  has the property that the equivariant embedding $\ess_g \into \R^2$
  is unique, up to a change of coordinates in $\R^2$ by a cubical
  automorphism. The action on $\R^2$ is always by a glide reflection,
  for this $g$. Other elements have characteristic sets that may embed
  in more than one way, with $g$ acting on $\R^2$ either as a translation
  or a glide reflection (depending on the embedding). 

  \subsection*{The staircase}

  Our goal in the rest of the paper will be to associate to each
  hyperbolic element $g$ a segment $\gamma$ such that
  $\widehat{\varphi}_\g(g) \geq 1$. Bavard Duality then will allow us to
  conclude that $\scl(g) \geq 1/24$. Here we illustrate one of the 
  difficulties in finding such segments. 

  Consider $\R^2$ with its standard integer cubing, and let
  $X$ be the subcomplex obtained by removing all vertices $(x,y) \in
  \Z^2$ with $y < x-1$ (see \figref{Staircase}). We will refer to $X$ as
  the \emph{staircase}. 
  \begin{figure}[htp!]
  \begin{center}
    \includegraphics{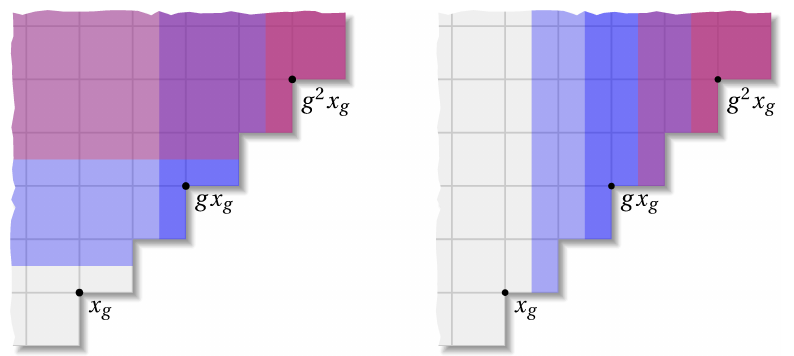}
  \end{center}
  \caption{Some tightly nested pairs in the staircase: $\set{H_1, H_2}$ in
  blue, $g \set{H_1, H_2}$ in red.} 
  \label{Fig:Staircase}
  \end{figure}

  Let $G=\gen{g}$, where $g$ is the restriction of the translation 
  $(x,y) \mapsto (x+2,y+2)$ to $X$. Note that $X=X_g = \ess_g$. Let $x_g =
  (0,0)$. Consider the two half-spaces 
  \[ 
    H_1 = \set{ (x,y) \in X \st y \ge 1/2} \quad \text{and} \quad H_2 =
    \set{ (x,y) \in X \st x \ge 3/2} 
  \] 
  shown in blue on the left hand side of \figref{Staircase}. The set
  $\gamma = \set{H_1, H_2}$ is a segment in $[x_g,gx_g]$ (recall that
  this means $\g$ is tightly nested). For any positive integer $n$,
  $g^nH_1$ and $H_2$ are transverse, so $\gamma$ and $g^n \g$ overlap. It
  follows that $c_\g(g^n) = 1$ for all $n$, which means that 
  $\widehat{\varphi}_\g(g) \leq 0$. 

  A better choice of segment $\gamma \subset [x_g, gx_g]$ is shown on the
  right hand side of \figref{Staircase}. The half-space $H_1$ has been
  replaced by $\set{(x,y) \in X \st x \ge 1/2}$. In this case, $\g$ and
  $g\g$ do not overlap, and in fact $c_\g(g^n) = n$ for all positive $n$. 

  This example indicates that from the point of view of an equivariant
  Euclidean embedding, one should choose a segment $\gamma$ which lies
  in a single coordinate direction in $\R^d$ to ensure that
  $c_\g(g^n)$ grows linearly with $n$. (Keeping $c_{\bg}(g^n)$
  bounded is a much more serious hurdle to be dealt with in
  Sections \ref{Sec:TightlyNested} and \ref{Sec:Last}.) It is for this
  reason that we required one of the chains in the Dilworth partition
  to be maximal in \lemref{EquiDilworth}, leading to property 3 in
  \propref{TautEmbedding}. This property ensures that in at least one
  coordinate direction of $\R^d$, consecutive half-spaces in $\R^d$
  are tightly nested in $\ess_g$, and therefore define segments in
  $\ess_g$. 

\section{Quadrants} 

  \label{Sec:Quadrants}

  In this section we present two basic tools for working with equivariant
  Euclidean embeddings: the Quadrant Lemma and the Elbow Lemma. They are
  useful in determining which cubes in $\R^d$ are occupied by $\ess_g$. Let
  $x_i$ and $x_j$ be coordinates of $\R^d$. We will denote by $p_{ij} \from
  \R^d \to \R^2$ the projection of $\R^d$ onto the $x_ix_j$--coordinate
  plane.

  Consider a $\gen{g}$--equivariant embedding $\ess_g \into \R^d$, where $d
  = \dim \ess_g$. Recall that via this embedding we identify elements of
  $\axis_g^+$ with their corresponding half-spaces $\half^+$ in
  $\half(\R^d)$. We will generally suppress the embedding itself and will
  treat $\ess_g$ as a subcomplex of $\R^d$. 

  \begin{remarks}\label{Rem:Square}

    (a) Recall from \lemref{Crossing} that if $H, H' \in \axis_g^+$ then $H$
    and $H'$ are transverse in $X$ if and only if they are transverse in
    $\ess_g$. When this occurs, they will also be transverse in $\R^d$, but
    not conversely. 

    (b) Expressing these two half-spaces as $H^i_n$ and $H^j_m$, the
    subcomplex $p_{ij}(\ess_g)$ of $\R^2$ contains the square $[n,n+1]
    \times [m, m+1]$ if and only if $H^i_n$ and $H^j_m$ are transverse in
    $\ess_g$. To see this, note that the latter occurs if and only if
    $\partial H^i_n$ and $\partial H^j_m$ cross in some square in $\ess_g
    \subseteq \R^d$. 
    If they cross in the square $S$, then the image of $S$ in $\R^2$
    is a cube which is dual to both $H^i_n$ and $H^j_m$, which must 
    be the square $[n, n+1] \times [m, m+1]$. 

    (c) If $H, H' \in \axis_g^+$ then $H$ and $H'$ are (tightly) nested
    in $X$ if and only if they are (tightly) nested in $\ess_g$. If they
    are nested in $\R^d$ then they are nested in $\ess_g$, but not
    conversely. There is no a priori relation between being tightly
    nested in $X$ and being tightly nested in $\R^d$. Half-spaces $H$ and
    $H'$ may be tightly nested in $X$ and not tightly nested in $\R^d$,
    and vice versa. 

  \end{remarks} 

  \begin{definition}

    A \emph{quadrant} in $\R^d$ is an open set of the form 
    \[ 
      \set{(x_1, \dotsc, x_d) \st x_i < n \text{ and } x_j > m }
    \]
    where $i \not= j$ and $m,n \in \Z$. Often, one of
    the coordinates $x_i$ or $x_j$ will be designated as the
    \emph{horizontal} coordinate. If $x_i$ is horizontal, then the quadrant
    above is called a \emph{northwest quadrant}, and if $x_j$ is
    horizontal, it is called a \emph{southeast quadrant}. 

  \end{definition}

  \begin{lemma}[Quadrant Lemma] \label{Lem:Quadrant}

    Let $H^i_n, H^j_m \in \half^+$ be half-spaces with $i \not= j$
    and suppose $x_i$ is horizontal. Then one of the following holds:  
      \begin{enumerate}[label=\textup{(\arabic*)}]
      \item $H^i_n$ and $H^j_m$ are transverse in $\ess_g$;
      \item $H^i_n \supset H^j_m$ in $\ess_g$ and $\ess_g$ is disjoint from the 
        northwest quadrant $\set{ x_i < n+1, \  x_j > m}$;
      \item $H^i_n \subset H^j_m$ in $\ess_g$ and $\ess_g$ is disjoint from the
        southeast quadrant $\set{ x_i > n, \  x_j < m+1 }$.
      \end{enumerate}
    The quadrant in case 2 or 3 that is disjoint from $\ess_g$ will be
    called the \emph{quadrant generated by $H^i_n$ and $H^j_m$}. 

  \end{lemma}

  Put another way, if $p_{ij}(\ess_g)$ does not contain the square $[n,
  n+1] \times [m, m+1]$, then it does not meet the quadrant generated by
  that square; see \figref{Quadrants}. 

  Whenever $H \in \half^i$, $K \in \half^j$ are nested in $\ess_g$ with
  $i\not= j$, denote by $Q(H,K)$ the quadrant generated by this pair of
  half-spaces. By definition, it is always disjoint from $\ess_g$. 

  \begin{proof}

    If the first alternative does not hold, then the corresponding
    half-spaces in $\ess_g$ are nested, by Remark \ref{Rem:TransNest}. That
    is, one of $H^i_n \cap \ess_g$, $H^j_m \cap \ess_g$ contains 
    the other. Suppose $H^i_n \cap \ess_g$ contains $H^j_m \cap \ess_g$. If a
    vertex $v = (v_1, \dotsc, v_d)$ of $\ess_g$ satisfies $v_j \geq m+1$ then
    $v \in H^j_m \cap \ess_g$, so $v \in H^i_n$. Hence $v_i \geq n+1$,
    showing that $v \not\in \set{ x_i \leq n, \  x_j \geq m+1}$. Thus the
    second alternative holds. Similarly, if $H^j_m \cap \ess_g$ contains
    $H^i_n \cap \ess_g$, then the third alternative holds. \qedhere

  \end{proof}
  
  \begin{figure}[htp!]
  \begin{center}
    \includegraphics{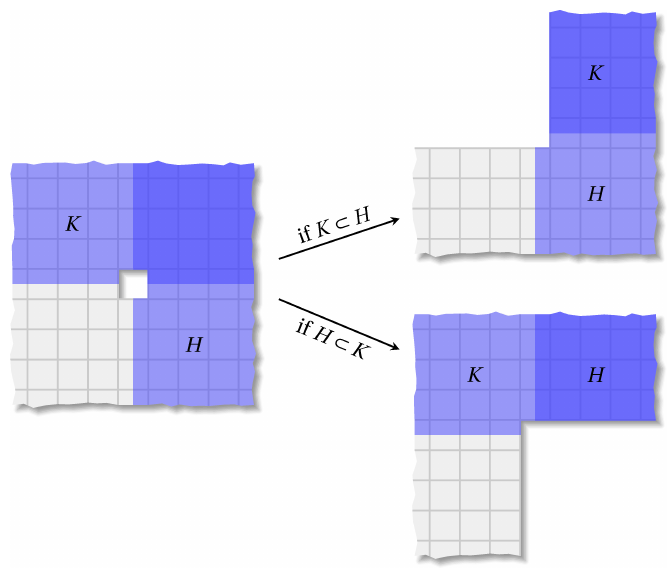}
  \end{center}
  \caption{The Quadrant Lemma: if $\ess_g$ avoids the interior of a
    square, it also avoids a northwest or southeast quadrant.}
  \label{Fig:Quadrants}
  \end{figure}

  \begin{lemma}\label{Lem:SameWay}

    Suppose $H \in \half^i$ and $K, K' \in \axis_g^+ -
    \half^i$ are such that $K, K'$ are tightly nested and the
    pairs $H, K$ and $H, K'$ are nested in $\ess_g$. Let $x_i$ be
    horizontal. Then the quadrants $Q(H,K)$ and $Q(H, K')$ both face
    northwest or both face southeast. 

  \end{lemma}

  \begin{proof}

    Suppose without loss of generality that $K \subset K'$. If $Q(H, K)$
    faces northwest and $Q(H,K')$ faces southeast, then $K \subset H$ and
    $H \subset K'$ by the Quadrant Lemma. Now $H$ violates the assumption
    that $K, K'$ are tightly nested. If $Q(H, K')$ faces northwest and
    $Q(H, K)$ faces southeast, then $K' \subset H$ and $H \subset
    K$. Hence $K' \subset K$, a contradiction. 
    \qedhere

  \end{proof}

  \begin{lemma}[Elbow Lemma]\label{Lem:Elbow}

    Suppose $H^i_n \subset H^j_m$ are tightly nested in $\ess_g$ where $i\not=
    j$. Then the edges $\{n\} \times [m, m+1]$ and $[n, n+1] \times
    \{m+1\}$ are contained in $p_{ij}(\ess_g)$.

  \end{lemma}

  The two edges form an ``elbow'' at the corner of the quadrant $Q(H^i_n,
  H^j_m) = \set{ x_i > n, \ x_j < m+1}$ (and $\ess_g$ avoids this quadrant,
  by the Quadrant Lemma).

  \begin{proof}

    Designate $x_i$ as the horizontal coordinate. We consider the edge
    $\{n\} \times [m, m+1]$ (the other case being entirely similar). 

    If the square $[n-1,n] \times [m, m+1]$ is in $p_{ij}(\ess_g)$, then
    so is the edge $\{n\} \times [m, m+1]$ and we are done. If not, then
    the half-spaces $H^i_{n-1}$ and $H^j_m$ are nested in $\ess_g$. We
    cannot have $H^i_{n-1} \subset H^j_m$ in $\ess_g$, because $H^i_n
    \subset H^i_{n-1}$ and $H^i_n, H^j_m$ are tightly nested. Therefore,
    $H^j_m \subset H^i_{n-1}$ in $\ess_g$. By the Quadrant Lemma, the
    northwest quadrant generated by the square $[n-1,n] \times [m,m+1]$
    is disjoint from $p_{ij}(\ess_g)$. Similarly, since $H^i_n \subset
    H^j_m$, the southeast quadrant generated by the square $[n,
    n+1]\times [m, m+1]$ is also disjoint from $p_{ij}(\ess_g)$. The edge
    $\{n\} \times [m, m+1]$ now provides the only passage across the
    strip $\R \times [m, m+1]$. It must be in $p_{ij}(\ess_g)$, or
    $\ess_g$ could not contain an axis. \qedhere

  \end{proof}

  \begin{remark}  

    The Quadrant Lemma and the Elbow Lemma do not use the fact the
    embedding $\ess_g \into \R^d$ is equivariant. These results hold
    (with $\ess_g$ replaced with $Y$) whenever $Y$ is a convex
    subcomplex of a CAT(0) cube complex $X$ and there is a Euclidean
    embedding $Y \into \R^d$ that induces a bijection between
    $\half(Y)$ and $\half(\R^d)$. 

  \end{remark}
  
\section{RAAG-like actions on cube complexes} 

  \label{Sec:RAAG-like}

  Recall from \secref{Prelim} that every right-angled Artin group
  $A_{\Gamma}$ acts on a CAT(0) cube complex $X_{\Gamma}$, and that the
  oriented edges of $X_{\Gamma}$ admit an $A_{\Gamma}$--invariant labeling
  by the generators and their inverses. Also, there is an induced
  $A_{\Gamma}$--invariant labeling of the half-spaces of $X_{\Gamma}$. 

  As noted earlier, properties of the half-space labeling lead to
  many useful observations about $X_{\Gamma}$ and its 
  $A_{\Gamma}$--action. The definition below is based on some of these
  properties of $X_{\Gamma}$. 

  \begin{definition}\label{Def:raaglike}
    Let $X$ be a CAT(0) cube complex with an action by $G$. The action is
    \emph{RAAG-like} if it is without inversion and also satisfies: 
    \begin{enumerate}[label=\textup{(\roman*)}]
      \item there do not exist $H \in \half(X)$, $g \in G$
        with $H \trans gH$, \label{raag1}
      \item there do not exist tightly nested $H, H' \in
        \half(X)$, $g \in G$ with $H \trans gH'$, \label{raag2}
      \item there do not exist $H \in \half(X)$, $g \in
        G$ with $H$ and $g\bH$ tightly nested. \label{raag3}
    \end{enumerate}

  When the $G$--action on $X$ is understood, we may simply say that
  $X$ is RAAG-like. 

  \end{definition}

  \begin{remark} \label{Rem:Subdivision} 

    If one has a $G$--action on $X$ with an inversion, it is customary to
    perform a cubical subdivision to obtain an action without
    inversion. We note here that the resulting action will never be
    RAAG-like, since it will violate property
    \ref{Def:raaglike}\ref{raag3}. 

  \end{remark}

  \begin{lemma}

    For every simplicial graph $\Gamma$, the action of $A_{\Gamma}$ on
    $X_{\Gamma}$ is RAAG-like. 

  \end{lemma}

  \begin{proof}

    We have already observed in \secref{Prelim} that $A_{\Gamma}$ acts
    without inversion on $X_{\Gamma}$. We have also observed that since 
    boundaries of squares in $X_{\Gamma}$ are labeled by commutators
    $[v,w]$ with $v \not= w$, no two half-spaces in $X_{\Gamma}$ with the
    same label can cross. Property \ref{raag1} follows immediately. 

    For \ref{raag2}, suppose $H, H'$ are tightly nested half-spaces in
    $X_{\Gamma}$. Then there is a vertex $x \in X_{\Gamma}$ and a pair of
    edges $e, e'$ both incident to $x$, such that $e$ is dual to $H$
    and $e'$ is dual to $H'$ (modulo orientations). Since $H$ and $H'$
    do not cross, the edges $e$ and $e'$ are not in the boundary of a
    common square; hence their labels do not commute in
    $A_{\Gamma}$. It follows that no two half-spaces bearing these
    labels (or their inverses) can cross. In particular, $H$ and $gH'$
    cannot cross for any $g \in A_{\Gamma}$. 

    For \ref{raag3}, suppose $H$ and $g\bH$ are tightly nested for some
    $H \in \half(X_{\Gamma})$, $g \in A_{\Gamma}$. Switching $H$ and
    $\bH$ if necessary, we may assume that $H \subset g\bH$. Since they
    are tightly nested, there is a pair of (oriented) edges $e, e'$ with common 
    initial vertex $x$ such that $e$ is dual to $H$ and $e'$ is dual to
    $gH$. Then $e$ and $e'$ bear the same label $v$, since the half-space
    labeling is $A_{\Gamma}$--invariant. However, vertices in
    $X_{\Gamma}$ have exactly one edge incident to them with any given 
    label (being lifts of the same oriented edge of $K(A_{\Gamma},1)$ at the same
    initial vertex). This contradiction establishes property
    \ref{raag3}. \qedhere 

  \end{proof}

  \begin{remark} \label{Rem:Special} 

    The properties of \defref{raaglike} correspond precisely to the
    defining properties of \emph{special cube complexes} due to 
    Haglund and Wise \cite{HaglundWise}, as enumerated in \cite{Wise}. More
    specifically, if $G$ acts freely on a CAT(0) cube complex $X$, then
    the action is RAAG-like if and only if $X/G$ is special. 

    The properties correspond as follows. Property \ref{raag1} means that
    immersed hyperplanes in $X/G$ are embedded (and hence can simply be
    called hyperplanes). $G$ acting on $X$ without inversion means that
    all hyperplanes in $X/G$ are two-sided. Property \ref{raag2} means
    that pairs of hyperplanes in $X/G$ do not inter-osculate. Property
    \ref{raag3} means that hyperplanes in $X/G$ do not self-osculate. 

  \end{remark}

  \begin{remark} \label{Rem:NonTransverse} 

    Note that \defref{raaglike}\ref{raag1} in particular means that
    the action of $G$ on $X$ is non-transverse. Therefore, for any
    hyperbolic element $g \in G$, the action of $\gen{g}$ on $\ess_g$
    is non-transverse. Hence, by \propref{Equality}, $\ess_g =
    \mess_g$ for all hyperbolic elements $g \in G$.

  \end{remark}

  \section{Tightly nested segments in the essential characteristic set}
  
  \label{Sec:TightlyNested}

  In \secref{Quadrants}, we presented some general tools for studying
  equivariant Euclidean embeddings of $\ess_g$. Here we develop more
  specialized results to be used in proving the main theorem. Generally
  speaking, these results deal with situations where there is a tightly
  nested segment $\sigma \subset \axis_g^+$ in one coordinate direction
  $\half^i$, and an element $f \in G$ such that $f\bsigma \subset
  \axis_g^+$. 
  
  For the rest of this section and the next section, we will assume that
  $X$ is a CAT(0) cube complex with a RAAG-like $G$--action. 

  Fix a hyperbolic element $g \in G$ and apply \propref{TautEmbedding} to
  obtain a taut $\gen{g}$--equivariant embedding $\ess_g \into \R^d$.
  Recall that a cube of maximal dimension $C \subset
  \ess_g$ is mapped to $[0,1]^d 
  \subset \R^d$, and we identify $\axis_g^+$ with $\half^+(\R^d) =
  \half^1\sqcup\cdots\sqcup \half^d$. The set of half-spaces in $\axis_g^+$
  dual to $C$ is denoted $A$, and $[A,gA) \cap \half^1 =
  \set{H_0^1,\ldots,H_n^1}$ is a tightly nested segment in $\axis_g^+$.
  Since $\gen{g}$ acts non-transversely on $\ess_g$, we also have $[A, gA)
  = [o, go]$, where $o$ denotes the origin in $\R^d$. 

  \begin{remark}\label{Rem:ExtendedQuad} 

    Since the action is assumed to be RAAG-like, property
    \ref{Def:raaglike}\ref{raag1} implies that if $H\in \half^i$ and $hH
    \in \half^j$ with $i \not= j$ for some $H \in \axis_g^+$, $h \in G$,
    then the quadrant $Q(H, hH)$ exists. Property
    \ref{Def:raaglike}\ref{raag2} implies that if in addition $H$ and $
    H' \in \half^i$ are tightly nested in $\ess_g$, then the quadrant
    $Q(H', hH)$ exists. (Recall that, by definition, $Q(H,K)$ is always
    disjoint from $\ess_g$.)  

    When discussing a quadrant $Q$ of the form $Q(H, hH)$, if $H, H' \in
    \half^i$ are tightly nested, the quadrant $Q(H', hH)$ faces the same
    way as $Q$ by \lemref{SameWay}. It either properly contains $Q$ or is
    properly contained in $Q$. If the former occurs, we may refer to
    $Q(H', hH)$ as an \emph{extended quadrant} for $Q$. 

  \end{remark}

    The first two results below will be used to generate
    contradictions. The first lemma states that when a certain
    configuration occurs (involving both a northwest and southeast
    quadrant) then there is a tightly nested segment that is forced to
    overlap with a copy of its reverse. The second lemma says that
    the latter event is impossible. Several of the arguments in the remainder of
    the paper have the goal of showing that a quadrant faces a
    particular way (northwest or southeast), with the aim of creating the
    forbidden configuration, and thereby a contradiction. 

  \begin{lemma} \label{Lem:ForcingOverlap}
    
    Let $\sigma=\set{K_0, \dotsc, K_m} \subset \half^i$ be tightly nested
    in $\ess_g$ and suppose that $f \bsigma \subset \axis_g^+$ for some $f \in
    G$. Let $x_i$ be horizontal. Suppose there exist $j \leq j'$ such that
    $f\bK_{j}, f\bK_{j'} \notin \half^i$ and $Q(K_{j},f\bK_{j})$ faces
    northwest while $Q(K_{j'}, f\bK_{j'})$ faces southeast. Then there is a
    non-trivial subsegment $\alpha \subset \sigma$ such that $f\balpha
    \subset \half^i$ and $\alpha, f\balpha$ overlap. 

  \end{lemma}

  \begin{proof}

    First note that if $j = j'$ then $\ess_g$ avoids both of the quadrants
    \[ \set{x_i < n+1, x_j > m} \quad \text{and} \quad \set{x_i > n, x_j <
    m+1} \] for some $n, m \in \Z$. But then $\ess_g$ avoids the set $\set{
    n < x_i < n+1}$ and cannot contain an axis for $g$. Thus $j < j'$. 

    For any index ${k}$, the quadrant $Q(K_{{k}}, f\bK_{{k}})$ is
    defined if and only if $f\bK_{{k}} \not\in \half^i$, by
    \remref{ExtendedQuad}. We may choose $j, j'$ to be an
    \emph{innermost} pair having the stated properties. Then, for any
    ${k}$ between $j$ and $j'$, we have $f\bK_{{k}} \in \half^i$. 

    Since $K_{j'-1}, K_{j'}$ are tightly nested there is an extended
    quadrant $Q(K_{j'-1}, f\bK_{j'})$ which faces southeast (cf.
    \remref{ExtendedQuad}). There is also an extended northwest quadrant
    $Q(K_{j+1}, f\bK_j)$, since $K_j, K_{j+1}$ are tightly nested. 

    If $j' = j+1$ then $f\bK_j$ and $f\bK_{j'}$ are tightly nested and
    \lemref{SameWay} says that both quadrants $Q(K_{j},f\bK_{j})$ and
    $Q(K_{j}, f\bK_{j'}) = Q(K_{j'-1}, f\bK_{j'})$ face the same way.
    However, these face northwest and southeast respectively. Therefore,
    $j' > j+1$ and the segment $\alpha = \set{K_{j+1}, \dotsc, K_{j'-1}}$
    is non-trivial. 

    Note that $f \balpha \subset \half^i$ by the choice of $j, j'$. We
    proceed now to use the Elbow Lemma (\ref{Lem:Elbow}) to constrain the
    location of $f\balpha$ along $\half^i$. In coordinates we
    have $K_j = H^i_a$ and $K_{j'} = H^i_b$ for some integers $a < b$, and
    \[ \alpha = \set{K_{j+1},\ldots,K_{j'-1}} = \set{ H^i_{a+1}, \dots,
    H^i_{b-1}}.\] Write $f\balpha = \set{H^i_c,  \dots, H^i_{c+\abs{\alpha}
    - 1} }$ for some $c \in \Z$.  

    Let $f\bK_{j'} = H^{i'}_e \in \half^{i'}$ where $i' \not= i$ and $e \in
    \Z$. Applying the Elbow Lemma to the tightly nested pair $\set{
    f\bK_{j'-1}, f\bK_{j'}} = \set{ H^i_c, H^{i'}_e}$, we find that the
    edge $ \set{c} \times [e, e+1]$ lies in $p_{i i'}(\ess_g)$. Since
    $\ess_g$ avoids the quadrant $Q(K_{j'-1}, f\bK_{j'}) = \set{ x_i > b -
    1, \ x_{i'} < e+1}$, we conclude that $c \leq b - 1$. See
    \figref{Overlap}. 

    \begin{figure}[htp!]
    \begin{center}
    \includegraphics{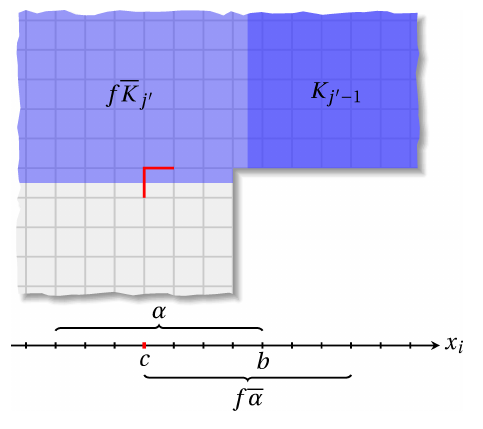}
    \end{center}
    \caption{The vertical position of the elbow is aligned with the
      top of the quadrant $Q(K_{j'-1},f\bK_{j'})$ as shown. The
      horizontal position is aligned with the left end of $f
      \balpha$. Since the elbow is outside the quadrant, $f\balpha$
      cannot be entirely to the right of $\alpha$.} \label{Fig:Overlap}
    \end{figure}

    Now redefine $i'$ and $e$ such that $f\bK_j = H^{i'}_e \in \half^{i'}$
    (with $i' \not= i$). Applying the Elbow Lemma to the tightly nested
    pair $\set{f\bK_j, f\bK_{j+1}} = \set{ H^{i'}_e,
    H^i_{c+\abs{\alpha}-1} }$, we find that $p_{ii'}(\ess_g)$ contains the
    edge $\set{c+\abs{\alpha}} \times [e, e+1]$. Now $\ess_g$ avoids the
    quadrant $Q(K_{j+1}, f\bK_j) = \set{x_i < a+2 , \ x_{i'} > e }$, and
    therefore $c+\abs{\alpha} \geq a+2$. 

    The inequalities $c \leq b - 1$ and $c + \abs{\alpha} \geq a+2$ say
    precisely that $\alpha$ and $f\balpha$ overlap. \qedhere

  \end{proof}
  
  The conclusion of the preceding lemma leads directly to a
  contradiction: 

  \begin{lemma} \label{Lem:NoTouching}
    
    Let $\alpha \subset \half^i$ be tightly nested in $\ess_g$ and
    suppose that $h \balpha \subset \half^i$ for some $h\in
    G$. Then $\alpha$ and $h\balpha$ cannot overlap. 

  \end{lemma}

  \begin{proof}

    Write $\alpha = \set{H_a, \dotsc, H_{a+k}}$ and $h \balpha =
    \set{H_{b-k}, \dotsc, H_{b}}$ for some $a, b \in \Z$. Then, $h
    H_{a+j} = \bH_{b-j}$ for each 
    $j$. The transformation $a+j \mapsto b-j$ either fixes $c$ or
    exchanges $c$ and $c+1$, for some $c\in \Z$. If $\alpha$ and $h
    \balpha$ overlap then $H_c$ (and $H_{c+1}$ in the second case) are in
    $\alpha \cap h\balpha$. In the first case $h$ inverts $H_c$, contrary
    to the assumption that $G$ acts on $X$ without inversion. In the
    second case $h\bH_{c} = H_{c+1}$, violating property
    \ref{Def:raaglike}\ref{raag3}. Thus $\alpha$ and $h\balpha$ cannot
    overlap. \qedhere 

  \end{proof}

  The next results perform a technical step that will be used repeatedly in
  the course of proving the main theorem. They also yield corollaries
  showing that under certain conditions, quadrants face in particular
  directions. 

  \begin{lemma}[Southeast quadrant shifting] \label{Lem:LeftQuad}

    Let $\sigma=\set{K_0, \dotsc, K_m} \subset \half^i$ be
    tightly nested in $\ess_g$ and suppose there is an $f\in G$ such that $f
    \bsigma \subset \axis_g^+$ and $f\bsigma \not\subset \half^i$. Let
    $x_i$ be horizontal and let $k$ be the smallest index such that $f
    \bK_k \not\in \half^i$. 
    Suppose the quadrant $Q(K_k, f\bK_k)$ faces southeast, so that 
    \[ 
      Q(K_k, f\bK_k) \ = \ \set{x_i > a + k, \ x_j < b+1}
    \]
    for some $j \not= i$, $a, b\in \Z$. Then 
    \begin{enumerate}[label=\textup{(\arabic*)}]
      \item $\ess_g$ also avoids the larger quadrant 
        $Q \ = \ \set{ x_i > a - k, \ x_j < b+1}.$ \label{leftquad1} 
      \item If\/ $k > 0$ then $f\bK_0 \in \half^i$ and
        $K_0 \subset f \bK_0$. \label{leftquad2}
    \end{enumerate}

  \end{lemma}

  \begin{proof}

    If $k=0$ then $Q = Q(K_k, f\bK_k)$ and there is nothing to prove, so
    assume that $k > 0$.  
    It is implicit from the description of $Q(K_k, f\bK_k)$ that $K_k =
    H^i_{a+k}$ and $f\bK_k = H^j_b$. Let $\alpha = \set{K_0, \dotsc,
      K_{k-1}}$ be the initial segment of $\sigma$ before $K_k$ and note
    that $f \balpha \subset \half^i$. Writing $f \balpha = \set{H^i_{c-k},
      \dotsc, H^i_{c-1}}$ for the appropriate $c \in \Z$, we have
    $f\bK_{k-1} = H^i_{c-k}$. 

    Applying the Elbow Lemma to the tightly nested pair $\set{ f
      \bK_{k-1}, f \bK_k} = \set{H^i_{c-k}, H^j_b}$, we find that
    $p_{ij}(\ess_g)$ contains the edge $e = \set{c-k} \times [b, b+1]$. 

    Since $e$ avoids the quadrant $Q(K_k, f\bK_k)$, we must have $c-k
    \leq a+k$. In fact, $e$ avoids the extended quadrant $Q(K_{k-1},
    f\bK_k) = \set{x_i > a + k - 1, \ x_j < b+1}$ by
    \remref{ExtendedQuad}, and so $c-k < a+k$. Thus $f \balpha$
    cannot be entirely to the right of $\alpha$ in $\half^i$. By
    \lemref{NoTouching} $f\balpha$ cannot overlap with $\alpha$ and so
    it must lie entirely to its left. That is, $c \leq a$. See 
    \figref{Shifted}. 

    \begin{figure}[htp!]
    \begin{center}
    \includegraphics{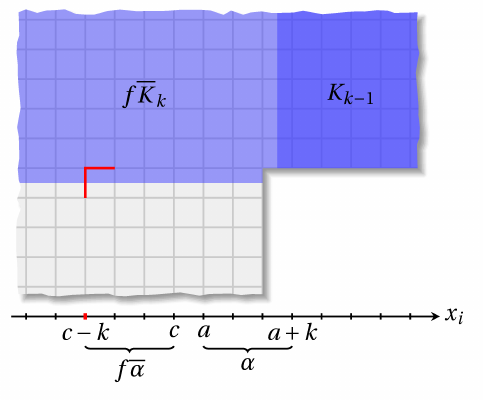}
    \end{center}
    \caption{The elbow is aligned with the left endpoint of
      $f\balpha$ and the top side of the quadrant $Q(K_{k-1},
      f\bK_k)$. Hence $f\balpha$ is not to the right of $\alpha$. Since
      $\alpha$ and $f\balpha$ cannot overlap, we have $c \leq a$. The
      elbow now generates a larger, ``shifted'' quadrant.} \label{Fig:Shifted}
    \end{figure}

    Now note that the quadrant generated by $f \bK_{k-1}$ and
    $f\bK_k$ (and avoided by $\ess_g$) is 
    \begin{align*}
      Q(f \bK_{k-1}, f \bK_k) \ &= \ \set{x_i > c - k, \ x_j < b+1} \\
      &\supseteq \ \set{x_i > a - k, \ x_j < b+1}, 
    \end{align*}
    proving \ref{leftquad1}. Finally, note that $K_0 = H^i_a$ and
    $f\bK_0 = H^i_{c-1}$, and \ref{leftquad2} is clear since $c -1 <
    a$. 
    \qedhere

  \end{proof}

  \begin{corollary} \label{Cor:LeftQuad}
    
    Let $\sigma=\set{K_0,\ldots,K_m} \subset \half^i$ be tightly nested
    in $\ess_g$ and suppose there is an $f \in G$ such that $f \bsigma
    \subset \axis_g^+$ and $f\bsigma \not\subset \half^i$. Let $x_i$ be
    horizontal and let $k$ be the smallest index such that $f\bK_k \notin
    \half^i$. Suppose there is a vertex $v$ in $\ess_g$ such that $v \in
    K_0$ and $v \notin f\bK_k$. Then the quadrant $Q(K_k, f\bK_k)$ faces
    northwest. 

  \end{corollary}
  
  \begin{proof}
    
    Set $K_0=H^i_a$ for some $a \in \Z$, so $K_k = H^i_{a+k}$. We assume
    $f\bK_k \notin \half^i$, so there exists $j \ne i$ and $b \in \Z$ such
    that $f\bK_k = H^j_b$. Suppose $Q(K_k,f\bK_k)$ faces southeast; that
    is: 
    \[ 
      Q(K_k,f\bK_k)= \set{x_i > a+k, \ x_j < b+1}.
    \] 
    By \lemref{LeftQuad}\ref{leftquad1}, $\ess_g$ also avoids the larger quadrant
    \[ 
      Q= \set{x_i > a-k, \ x_j < b+1}.
    \]
    Since $v \notin f\bK_k$, $v_j \le b$. But $v \in K_0$, so $v_i \ge a+1$.
    So $v \in Q$, which is a contradiction. Therefore, $Q(K_k,f\bK_k)$
    faces northwest. \qedhere
     
  \end{proof}
  
  The next two results are completely analogous to the previous two, with
  the same proofs: 

  \begin{lemma}[Northwest quadrant shifting] \label{Lem:RightQuad}

    Let $\sigma=\set{K_0, \dotsc, K_m} \subset \half^i$ be
    tightly nested in $\ess_g$ and suppose there is an $f\in G$ such that $f
    \bsigma \subset \axis_g^+$ and $f\bsigma \not\subset \half^i$. Let $x_i$
    be horizontal and let $k$ be the largest index such that $f \bK_k
    \not\in \half^i$. Suppose the quadrant $Q(K_k, f\bK_k)$ faces
    northwest, so that 
    \[ 
      Q(K_k, f\bK_k) \ = \ \set{x_i < a - (m-k), \ x_j > b}
    \]
    for some $j \not= i$, $a, b\in \Z$. Then
    \begin{enumerate}[label=\textup{(\arabic*)}]
      \item $\ess_g$ also avoids the larger quadrant 
      $Q \ = \ \set{ x_i < a+(m-k), \ x_j > b}$. \label{rightquad1}
      \item If\/ $k < m$ then $f\bK_m \in \half^i$ and
        $K_m \supset f \bK_m$. \label{rightquad2} \qed
    \end{enumerate}

  \end{lemma}

  \begin{corollary} \label{Cor:RightQuad}
    
    Let $\sigma=\set{K_0,\ldots,K_m} \subset \half^i$ be tightly nested
    in $\ess_g$ and suppose there is an $f \in G$ such that $f\bsigma
    \subset \axis_g^+$ and $f \bsigma \not\subset \half^i$. Let $x_i$ be
    horizontal and let $k$ be the largest index such that $f\bK_k \notin
    \half^i$. Suppose there is a vertex $v$ in $\ess_g$ such that $v
    \notin K_m$ and $v \in f\bK_k$. Then the quadrant $Q(K_k, f\bK_k)$
    faces southeast. \qed 

  \end{corollary}
   
\section{Proof of the main theorem} 

  \label{Sec:Last}

  Our goal in this section is to prove \thmref{Main} from the Introduction,
  which we restate here: 
  
  \begin{theorem} \label{Thm:Gap}

    Let $X$ be a CAT(0) cube complex with a RAAG-like action by $G$. Then
    $\scl(g) \geq 1/24$ for every hyperbolic element $g\in G$. 

  \end{theorem}

  We continue with the same notation as in the previous section. Fix a taut
  equivariant embedding $\ess_g \into \R^d$. Let $C$ be the cube in
  $\ess_g$ mapped to $[0,1]^d$ under the equivariant embedding. Let $A$ be
  the set of half-spaces in $\axis_g^+$ dual to $C$, so that $[A,gA)$ is a
  fundamental domain for the action of $\gen{g}$ on $\axis_g^+$. We have
  $[A,gA) = [o,go]$, where $o \in \R^d$ is the origin. Identify
  $\half(\R^d)$ with $\axis_g$. Recall that by property (3) of
  \propref{TautEmbedding}, the extended taut segment $[A,gA] \cap \half^1$
  is tightly nested in $\ess_g$. Write $[A,gA) \cap \half^1 =
  \set{H^1_0,\ldots,H^1_n}$. 
  
  Most of this section is devoted to finding a tightly nested subsegment
  $\gamma \subseteq [A,gA) \cap \half^1$ such that $\gamma > g\gamma$ and
  no copy of $\bg$ appears in $\axis_g^+$. Once we find such a $\gamma$,
  then \thmref{Gap} follows immediately; details are laid out in the proof
  at the end of this section. To arrange that $\gamma > g\gamma$, we may
  have to pass to a proper subsegment of $[A,gA) \cap \half^1$. On the
  other hand, if $\g$ is short, it is more likely for $\axis_g^+$ to
  contain a copy of $\bg$. Our approach, therefore, is to use a
  \emph{maximal $g$--nested segment}, defined below. 

  After discussing maximal $g$--nested segments, we proceed to show that
  for such a segment $\gamma$, no copy of $\bgamma$ can lie in $\axis_g^+$.
  First we prove Lemmas \ref{Lem:ForcingNorthwest} and
  \ref{Lem:ForcingSoutheast}, which are technical statements used to show
  that quadrants face a particular way, under appropriate conditions. Next
  comes \propref{Horizontal}, which states that if $h\bgamma \subset
  \axis_g^+$, then $h\bgamma$ cannot lie entirely in $\half^1$. Hence
  $h\bgamma$ contains half-spaces in $\axis_g^+$ outside of $\half^1$. Such
  a half-space generates a quadrant, by property
  \ref{Def:raaglike}\ref{raag1}. In Propositions \ref{Prop:Northwest} and
  \ref{Prop:Southeast} we prove that the first such quadrant faces
  northwest, and the last such quadrant faces southeast. In this way, the
  forbidden configuration of \lemref{ForcingOverlap} is created, resulting
  in a contradiction. 

  \subsection*{Maximal $g$--nested segments}

  \begin{definition}

    A subsegment $\gamma=\set{H_\ell^1,\ldots,H_r^1}$ of $[A,gA) \cap
    \half^1$ is said to be \emph{$g$--nested} if $\gamma > g \gamma$ in
    $\ess_g$. It is a \emph{maximal $g$--nested} segment if it is
    $g$--nested and is not properly contained in another $g$--nested
    subsegment of $[A,gA) \cap \half^1$. 
 
    \figref{Pure} shows an example where the full segment $\gamma = [A, gA)
    \cap \half^1$ is not $g$--nested; this is the segment of blue
    half-spaces labeled $a$, $c$, $e$. In this example, the subsegment
    consisting of the pair labeled $a$, $c$ is a maximal $g$--nested
    segment, as is the pair labeled $c$, $e$. 

    Note that for every $H \in \axis_g^+$ we have $H \supset gH$ in
    $\ess_g$ by Remark \ref{Rem:TransNest2} and Property
    \ref{Def:raaglike}\ref{raag1}. Thus every subsegment of $[A, gA) \cap
    \half^1$ of length $1$ is $g$--nested. It follows that every $H \in
    [A, gA) \cap \half^1$ is contained in a maximal $g$--nested segment. 

  \end{definition}
    
  \begin{lemma} \label{Lem:Nested}
      
      Let $\gamma=\set{H^1_\ell,\ldots,H^1_r}$ be a maximal $g$--nested
      subsegment of $[A, gA) \cap \half^1$. Then the following two
      statements hold: 
      \begin{enumerate}[label=\textup{(\arabic*)}]
        \item Either $\ell=0$ or $H^1_{\ell-1} \trans g^{-1}H^1_r$ in
        $\ess_g$.  \label{nested1}
        \item Either $r=n$ or $gH^1_\ell \trans H^1_{r+1}$ in
          $\ess_g$.  \label{nested2} 
      \end{enumerate}

  \end{lemma}
  
  \begin{proof}
    
    Suppose $\ell > 0$. If $H^1_{\ell-1}$ and $g^{-1}H^1_r$ are not
    transverse in $\ess_g$, then they are nested in $\ess_g$  by
    Remark \ref{Rem:TransNest}. Since $o \in g^{-1} H_r - H_{\ell
      -1}$, this means that $g^{-1}H^1_r \supset H^1_{\ell-1}$ in
    $\ess_g$, which is equivalent to $H^1_r \supset
    gH^1_{\ell-1}$. Let $\gamma' =
    \set{H^1_{\ell-1},\ldots,H^1_r}$. We have: 
    \[
      H^1_{\ell-1} \supset \cdots \supset H^1_r \supset gH^1_{\ell-1}
      \supset \cdots \supset gH^1_r.
    \] 
    So $\gamma'$ is $g$--nested and $\gamma'$ properly contains $\gamma$,
    violating maximality of $\gamma$. Similarly, if $r < n$ and
    $gH^1_\ell$ and $H^1_{r+1}$ are not transverse, then the segment
    $\set{H^1_\ell,\ldots,H^1_{r+1}}$ is $g$--nested and contains
    $\gamma$. \qedhere 
    
  \end{proof}
  
  We now proceed with the main steps of the proof of
  \thmref{Gap}. The primary goal is to show that a maximal $g$--nested
  segment in $[A,gA) \cap \half^1$ never appears in $\axis_g^+$ in
  the reverse direction. The next two lemmas are technical steps that
  are needed along the way.
    
  \begin{lemma} \label{Lem:ForcingNorthwest}
    
    Let $\gamma = \set{H^1_\ell,\ldots,H^1_r}$ be a maximal $g$--nested
    subsegment of $[A,gA) \cap \half^1$. Suppose $h \bgamma \subset
    \axis_g^+$ and $h\bgamma \not\subset \half^1$, for some $h \in G$. 
    Suppose $\ell > 0$. Let $x_1$ be horizontal and let $j$ be the
    smallest integer between $\ell$ and $r$ such that $h\bH^1_j \notin
    \half^1$. Then either the quadrant $Q(H^1_j,h\bH^1_j)$ faces
    northwest, or there is a vertex $v$ in $\ess_g$ such that $v \notin
    g^{-1}H^1_r$ and $v \in h\bH^1_\ell$. 

  \end{lemma}
  
  \begin{proof}
    
    By \lemref{Nested}, since $\ell >0$, $H^1_{\ell-1} \trans g^{-1} H^1_r$
    in $\ess_g$. Therefore, there exists a square $S$ in $\ess_g$ in which
    they cross. Let $v$ be the unique vertex of $S$ with $v_1 = \ell$ and
    $v \notin g^{-1}H^1_r$. We now show that $v \in h\bH^1_\ell$ under the
    assumption that $Q(H^1_j,h\bH^1_j)$ faces southeast.
    
    Let $h\bH^1_j = H^i_b$ for some $i \ne 1$ and $b \in \Z$. By
    assumption,  $\ess_g$ avoids the quadrant 
    \[
      Q(H^1_j,h\bH^1_j)  = \set{ x_1 > j, x_i < b+1}.
    \] 
    Since $\ell > 0$ and $j \ge \ell$, the half-spaces $H^1_{j-1}$ and
    $H^1_j$ are tightly nested. Thus, by \remref{ExtendedQuad}, $\ess_g$
    must further avoid the extended quadrant 
    \[
      Q(H^1_{j-1},h\bH^1_j) = \set{ x_1 > j-1, x_i < b+1}.
    \] 
    If $j=\ell$, then for $v$ to lie outside of $Q(H^1_{j-1},h\bH^1_j)$,
    we must have $v_i \ge b+1$, so $v \in H^i_b = h\bH^1_\ell$. 
    
    If $j > \ell$,
    then applying \lemref{LeftQuad}\ref{leftquad2} using 
    \[ 
      \set{K_0,\ldots,K_m} = \set{H^1_\ell,\ldots,H^1_r}, \quad 
      i=1, \quad k=j, \quad f=h, \quad \sigma=\gamma, \quad a=0
    \]
    we obtain that $h\bH^1_\ell \in \half^1$ and $h\bH^1_\ell \supset
    H^1_\ell$. In coordinates, this means that $h\bH^1_\ell = H^1_c$ for
    some $c < \ell=v_1$. Thus, $v \in H^1_c = h\bH^1_\ell$. \qedhere

  \end{proof}
  
  The next lemma is completely analogous to the previous one, with a
  similar proof.

  \begin{lemma} \label{Lem:ForcingSoutheast}
    
    Let $\gamma = \set{H^1_\ell,\ldots,H^1_r}$ be a maximal $g$--nested
    subsegment of $[A,gA) \cap \half^1$. Suppose $h \bgamma \subset
    \axis_g^+$ and $h\bgamma \not\subset \half^1$, for some $h \in
    G$. Suppose $r < n$. let $x_1$ be horizontal and let $j$ be the
    largest integer between $\ell$ and $r$ such that $h\bH^1_j \notin
    \half^1$. Then either the quadrant $Q(H^1_j,h\bH^1_j)$ faces
    southeast, or there is a vertex $v$ in $\ess_g$ such that $v \in
    gH^1_\ell$ and $v \notin h\bH^1_r$. \qed

  \end{lemma}
    
  The next three propositions will form the main body of the argument. The
  first one shows that if a reverse copy of a maximal $g$--nested segment
  appears in $\axis_g^+$, then it cannot lie entirely within
  $\half^1$.   

  \begin{proposition} \label{Prop:Horizontal}

    Let $\gamma = \set{H^1_\ell,\ldots,H^1_r}$ be a maximal $g$--nested
    subsegment of $[A,gA) \cap \half^1$. Suppose $h \bgamma \subset
    \axis_g^+$ for some $h \in G$, and that $h\bH^1_r \in [A,gA)$. Then
    $h\bgamma \not\subset \half^1$. 

  \end{proposition}

  \begin{proof}

    If not then $h\bgamma \subset \half^1$. Write $h\bgamma =
    \set{H^1_a, \dotsc, H^1_{a + \abs{\gamma}-1}}$, where $a \geq
    0$ because $h\bH^1_r \in [A, gA)$.  

    Since $\gamma$ and $h\bgamma$ cannot overlap (by Lemma
    \ref{Lem:NoTouching}), there are two possibilities for the location
    of $h\bgamma$ along $\half^1$. 

    The first case is that $a + \abs{\gamma} \leq \ell$ (i.e. $h\bgamma$
    is to the left of $\gamma$). In particular $\ell > 0$ and therefore
    $H^1_{\ell-1} \trans g^{-1}H^1_r$ in $\ess_g$, by
    \lemref{Nested}. Let $g^{-1}H^1_r = H^i_b$ for some $i \not= 1$,
    $b\in \Z$. Note that $b < 0$ because $g^{-1}H^1_r$ contains the
    origin $o$. The projection $p_{1i}(\ess_g)$ contains the square
    $[\ell - 1 , \ell ] \times [b, b+1]$, which is dual to both
    $H^1_{\ell -1}$ and $g^{-1}H^1_r$. Thus there is a vertex $v \in
    \ess_g$ such that $v_1 = \ell$ and $v_i = b$. 

    By property \ref{Def:raaglike}\ref{raag1} the half-spaces $g^{-1}H^1_r$
    and $h\bH^1_r = H^1_a$ are not transverse in $\ess_g$, and hence
    they generate a quadrant $Q$ disjoint from $\ess_g$. However, the
    quadrant $\set{x_1 > a, \ x_i < b+1}$ contains $v$ and $\set{x_1 < a+1,
    \ x_i > b}$ contains $o$. These are the two possibilities for $Q$ and
    thus we have a contradiction (since $v, o \in \ess_g$). 

    The second case is that $r < a$ (i.e. $h\bgamma$ is to the right of
    $\gamma$). Note that $a \leq n$ since $h\bH^1_r \in [A, gA)$. Hence $r
    < n$, and $gH^1_{\ell} \trans H^1_{r+1}$ in $\ess_g$ by
    \lemref{Nested}. Now redefine $i\not= 1$ and $b\in \Z$ such that
    $g\gamma \subset \half^i$ and $gH^1_{\ell} = H^i_b$. Then
    $p_{1i}(\ess_g)$ contains the square $[r+1, r+2] \times [b, b+1]$ dual
    to $H^1_{r+1}$ and $H^i_b$. Let $v\in \ess_g$ be a vertex such that
    $v_1 = r+1$ and $v_i = b+1$. 

    Let $x_1$ be horizontal. By property \ref{Def:raaglike}\ref{raag1} the
    half-spaces $gH^1_{\ell}$ and $h\bH^1_{\ell} = H^1_{a + \abs{\gamma} -
    1}$ are not transverse in $\ess_g$, and hence they generate a
    quadrant $Q$ disjoint from $\ess_g$. The northwest quadrant $\set{x_1 <
    a + \abs{\gamma}, \ x_i > b}$ contains $v$, and therefore cannot be
    disjoint from $\ess_g$. Thus $Q(gH^1_{\ell}, h\bH^1_{\ell}) =
    Q(h\bH^1_{\ell}, g H^1_{\ell})$ faces southeast. 

    Again using property \ref{Def:raaglike}\ref{raag1}, the half-spaces
    $gH^1_r$ and $h \bH^1_r$ are not transverse in $\ess_g$ and generate
    a quadrant $Q(gH^1_r, h \bH^1_r) = Q(h \bH^1_r, g
      H^1_r)$ disjoint from 
    $\ess_g$. If it faces southeast then it is the quadrant $\set{x_1 > a,
    \ x_i < b + \abs{\gamma} }$. We have $g o \not\in g H^1_r$ because $o
    \not\in H^1_r$, and $go \in h\bH^1_r$ because
      $h\bH^1_r \in [A, gA) = [o, go]$.
    Therefore $g o$ is in this southeast quadrant. Since $go \in \ess_g$,
    we conclude that $Q(h \bH^1_r, g H^1_r)$ faces northwest. 

    Now apply \lemref{ForcingOverlap} using 
    \[
      \set{K_0, \dotsc, K_m} = \set{h \bH^1_r, \dotsc, h
          \bH^1_{\ell}}, \quad 
      i = 1, \quad f = gh^{-1}, \quad j = 1, \quad j' = m
    \]
    to obtain a contradiction via \lemref{NoTouching}. \qedhere

  \end{proof}
  
  The next two propositions also deal with a reverse copy of a maximal
  $g$--nested segment in $\axis_g^+$. By the previous proposition,
  there must be a half-space in the segment which lies outside of
  $\half^1$. Such a half-space will generate a quadrant, by property
  \ref{Def:raaglike}\ref{raag1}. The two propositions say that the
  first such quadrant always faces northwest, and the last such
  quadrant always faces southeast. 

  \begin{proposition} \label{Prop:Northwest}

    Let $\gamma = \set{H^1_\ell,\ldots,H^1_r}$ be a maximal $g$--nested
    subsegment of\/ $[A,gA) \cap \half^1$. Suppose $h \bgamma \subset
    \axis_g^+$, $h \bH^1_r \in [A,gA)$, and $h\bgamma \not\subset \half^1$
    for some $h \in G$.  Let $x_1$ be horizontal. Let $j$ be the smallest
    integer between $\ell$ and $r$ such that $h\bH^1_j \notin \half^1$.
    Then the quadrant $Q(H^1_j,h\bH^1_j)$  faces northwest. 

  \end{proposition}
  
  \begin{proof}
    
    {\bf Case 1: $\ell=0$}

    In other words, $\gamma=\set{H^1_0,\ldots,H^1_r}$. Let $v$ be the
    vertex of $\ess_g$ with coordinates $v_1 = 1$ and $v_k = 0$ for all $k
    > 1$. Note that $v \in H^1_0$ and $H^1_0$ is the only element in
    $[A,gA)$ with this property. Therefore, since $h\bH^1_r \in [A,gA)$, if
    $v \in h\bH^1_r$, then we must have $h\bH^1_r = H^1_0$. But this
    contradicts that $\gamma$ and $h\bgamma$ cannot overlap by
    \lemref{NoTouching}, so $v \notin h\bH^1_r$. Since
    $h\bH^1_r \supset h\bH^1_j$, $v \notin
    h\bH^1_j$. Now apply \corref{LeftQuad} using 
    \[ 
      \set{K_0, \dotsc, K_m} = \set{H^1_0, \dotsc, H^1_r}, \quad i=1,
      \quad f=h, \quad K_k = H^1_j,
    \] 
    to obtain that $Q(H^1_j,h\bH^1_j)$ must face northwest. 

    {\bf Case 2: $\ell>0$}
    
    We will assume $Q(H^1_j,h\bH^1_j)$ faces southeast and derive a
    contradiction. By \lemref{ForcingNorthwest}, there exists a vertex $v$
    in $\ess_g$ such that $v \notin g^{-1}H^1_r$ and $v \in h\bH^1_\ell$.
    Note for any $j$ between $\ell$ and $r$, $h\bH^1_\ell \subset
    h\bH^1_j$, so $v \in h\bH^1_j$.
    
    Let $i$ be the coordinate with $g^{-1}\gamma \subset \half^i$. If
    $h\bH^1_j \in \half^i$, then $g^{-1}H^1_r$ are $h\bH^1_j$ are parallel
    in $\R^d$, and hence are nested in $\ess_g$. 
    Since $h\bH^1_r \in [A, gA)$, we have $o \not\in h\bH^1_r$. Then,
    since $j < r$ we have $h\bH^1_r \supset h \bH^1_j$, and therefore
    $o \not\in h\bH^1_j$. However, $o \in g^{-1}H^1_r$ and therefore
    $g^{-1}H^1_r \supset h\bH^1_j$. This contradicts the existence of
    $v$. Therefore, we may assume $h\bH^1_j \notin \half^i$. 

    We now forget coordinate $x_1$ and designate $x_i$ to be the horizontal
    coordinate. Since $h\bgamma$ is not entirely contained in $\half^i$,
    there is a largest integer $j'$ between $\ell$ and $r$ such that
    $h\bH^1_{j'} \notin \half^i$. Since $v \notin g^{-1}H^1_r$ and $v \in
    h\bH^1_{j'}$, by \corref{RightQuad} using \[ \set{K_0,\ldots,K_m} =
    \set{g^{-1}H^1_\ell,\ldots,g^{-1}H^1_r}, \quad f=hg, \quad K_k =
    g^{-1}H^1_{j'},
    \] 
    we obtain that $Q(g^{-1}H^1_{j'},h\bH^1_{j'})$ faces southeast. That
    is, $h\bH^1_{j'} \supset g^{-1}H^1_j$, but this is impossible since $o
    \in g^{-1}H^1_{j'}$ and $o \notin h\bH^1_{j'}$. This yields a
    contradiction under the assumption that $Q(H^1_j,h\bH^1_j)$ faces
    southeast, as desired.\qedhere 
    
  \end{proof}

  The next proposition is analogous to the previous one, but the
  situation is not entirely symmetric because of the assumption
  throughout that the largest half-space of $h\bgamma$ lies in $[A,
  gA)$. For this reason, the next proposition requires an independent
  proof. 
  
  \begin{proposition} \label{Prop:Southeast}

    Let $\gamma = \set{H^1_\ell,\ldots,H^1_r}$ be a maximal $g$--nested
    subsegment of\/ $[A,gA) \cap \half^1$. Suppose $h \bgamma \subset
    \axis_g^+$, $h \bH^1_r \in [A,gA)$, and $h\bgamma \not\subset
    \half^1$ for some $h \in G$. Let $x_1$ be horizontal. Let $j$
    be the largest integer between $\ell$ and $r$ such that $h\bH^1_j
    \notin \half^1$. Then the quadrant $Q(H^1_j,h\bH^1_j)$ faces
    southeast. 

  \end{proposition}
  
  \begin{proof}

    In the following, let $i$ and $i'$ be the coordinates with $g\gamma
    \subset \half^i$ and $h\bH^1_j \in \half^{i'}$. We will assume that
    $Q(H^1_j,h\bH^1_j)$ faces northwest, that is, $H^1_j \supset
    h\bH^1_j$, and derive a contradiction.
    
    {\bf Case 1: $r=n$} 
    
    In other words, $\gamma = \set{H^1_\ell,\ldots,H^1_n}$. 
    
    We first consider the sub-case that $j=n$. Recall that the extended
    segment \[ [A,gA] \cap \half^1 = \set{H_0^1, \ldots, H^1_n,
    H^1_{n+1}}\] is tightly nested; in particular, the pair
    $\set{H^1_n,H^1_{n+1}}$ is tightly nested. Therefore, by
    \remref{ExtendedQuad}, $H^1_{n+1}$ and $h\bH^1_n$ must also generate a
    quadrant that faces northwest; in other words, $H^1_{n+1} \supset
    h\bH^1_n$. Let $go$ be the translate of the origin by $g$. Since $[A,gA)
    = [o,go]$ and $H^1_{n+1} \in gA$, we must have $go \notin H^1_{n+1}$.
    Thus, $go \notin h\bH^1_n$, but this contradicts the assumption that
    $h\bH^1_n=h\bH^1_r \in [A,gA)$. 
    
    Now suppose $j<n$. By \lemref{RightQuad}\ref{rightquad2}, using 
    \[ 
      \set{K_0,\ldots,K_m} = \set{H^1_\ell,\ldots,H^1_n}, \quad i=1,
      \quad f=h, \quad K_k = H^1_j,
    \] 
    we obtain that $h\bH^1_n \in \half^1$ and $H^1_n \supset h\bH^1_n$. So
    we must have that $h\bH^1_n = H^1_b$, for some $b > n$. But this
    contradicts $h\bH^1_n \in [A,gA) \cap \half^1 =
    \set{H^1_0,\ldots,H^1_n}$.
    
    {\bf Case 2: $r<n$} 

    In this case, we can apply \lemref{ForcingSoutheast} to the assumption
    that $Q(H^1_j,h\bH^1_j)$ faces northwest, yielding a vertex $v$ in
    $\ess_g$ with $v \in gH^1_\ell$ and $ v \notin h\bH^1_r$. 
    
    If $j=r$ and $i=i'$, then $gH^1_\ell$ and $h\bH^1_r$ are parallel in
    $\R^d$ and hence are nested in $\ess_g$. Since $go \in h\bH^1_r$ and
    $go \notin gH^1_\ell$, $h\bH^1_r \supset gH^1_\ell$. But this
    contradicts the existence of $v$.

    In all other cases we claim $h\bH^1_r \notin \half^i$. This is
    true when $j=r$ and $i \ne i'$, since $h\bH^1_r = h\bH^1_j \in
    \half^{i'}$. In the situation that $j < r$, then by the choice of
    $j$, the suffix $\alpha = \set{H^1_{j+1},\ldots,H^1_r}$ has
    $h\balpha \subset \half^1$. But since $r<n$, $gH^1_\ell \trans
    H^1_{r+1}$ by \lemref{Nested}\ref{nested2}; in particular, since
    $gH^1_\ell \in \half^i$, we have $i \ne 1$. This shows that
    $h\bH^1_r \notin \half^i$. 
   
    We now forget coordinate $x_1$ and designate $x_i$ to be the horizontal
    coordinate. Let $j'$ be the smallest integer between $\ell$ and $r$
    such that $h\bH^1_{j'} \notin \half^i$. Such $j'$ exists since
    $h\bH^1_r \notin \half^i$. Our goal now is to use $go$ and $v$ to
    determine which ways the quadrants $Q(gH^1_{j'}, h\bH^1_{j'})$ and
    $Q(gH^1_r, h\bH^1_r)$ face. 
    
    Now set
    \[ 
      \set{K_0,\ldots,K_m} = \set{gH^1_\ell,\ldots,gH^1_r}, \quad f =
      hg^{-1}. 
    \] 
    Since $go \notin gH^1_r$ and $go \in h\bH^1_r$, by \corref{RightQuad},
    where $K_k=gH^1_r$, the quadrant $Q(gH^1_r,h\bH^1_r)$ faces southeast.
    On the other hand, since $v \in gH^1_\ell$ and $v \notin h\bH^1_{j'}$,
    by \corref{LeftQuad}, where $K_k=gH^1_{j'}$, the quadrant
    $Q(gH^1_{j'},h\bH^1_{j'})$ faces northwest. Since $j'$ is the smallest
    index between $\ell$ and $r$ for which $h\bH^1_{j'} \notin \half^i$,
    and $r$ is the largest index, the conclusion of \lemref{ForcingOverlap}
    yields a non-trivial subsegment $\alpha \subset g\gamma$ such that
    $h\balpha \subset \half^i$ and $\alpha$ and $h\alpha$ overlap. But this
    is impossible by \lemref{NoTouching}. This contradiction was obtained
    under the assumption that $Q(H^1_j,h\bH^1_j)$ faces northwest,
    concluding the proof. \qedhere 

  \end{proof}
  
  We now tie everything together for the proof of \thmref{Gap}.

  \begin{proof}[Proof of \thmref{Gap}]
    
    Given a hyperbolic element $g \in G$, fix a taut $\gen{g}$--equivariant
    embedding $\ess_g \into \R^d$ using
    \propref{TautEmbedding}, as discussed in the beginning of 
    \secref{RAAG-like}. Let $\gamma = \set{H^1_{\ell}, \dotsc,
      H^1_r} \subseteq [A, gA)$ be a maximal $g$--nested subsegment of $[A,
    gA) \cap \half^1$. 

    Suppose $h \bgamma \in \axis_g^+$ for some $h \in G$. Replacing $h
    \bgamma$ by a $\gen{g}$--translate if necessary, we can assume that
    $h\bH^1_r \in [A, gA)$. Declare $x_1$ to be the horizontal
    coordinate. 

    By \propref{Horizontal}, $h\bgamma$ cannot be entirely contained in
    $\half^1$. Let $j$ be the smallest index such that $h\bH^1_j
    \not\in \half^1$. Then, by \propref{Northwest}, the quadrant
    $Q(H^1_j, h\bH^1_j)$ faces northwest. Let $j'$ be the largest
    index such that $h\bH^1_{j'}\not\in \half^1$. By \propref{Southeast}
    the quadrant $Q(H^1_{j'}, h\bH^1_{j'})$ faces
    southeast. \lemref{ForcingOverlap} now provides a contradiction, via
    \lemref{NoTouching}. 

    Therefore, no copy of $\bgamma$ appears in $\axis_g^+$. Now consider
    the counting functions $c_\g$ and $c_{\bg}$ from
    \secref{RAAG-like}. We have $c_{\bg}(o, g^no) = 0$ for all $n >
    0$. Since $\g$ is $g$--nested we also have $c_\g(o, g^no) = n$ for $n
    > 0$. Choosing the basepoints $x_{g^n} = o$ for all such $n$ (noting
    that $X_{g} \subseteq X_{g^n}$), the resulting homogeneous
    quasimorphism $\widehat{\varphi}_\g$ has value $1$ on $g$. Since
    $\widehat{\varphi}_\g$ has defect at most $12$, by \lemref{Defect},
    Bavard Duality (\lemref{SCL}) tells us that $\scl(g) \geq
    1/24$. \qedhere 

  \end{proof}

  \small

\newcommand{\etalchar}[1]{$^{#1}$}
\def\cprime{$'$}
\providecommand{\bysame}{\leavevmode\hbox to3em{\hrulefill}\thinspace}
\providecommand{\MR}{\relax\ifhmode\unskip\space\fi MR }
\providecommand{\MRhref}[2]{%
  \href{http://www.ams.org/mathscinet-getitem?mr=#1}{#2}
}
\providecommand{\href}[2]{#2}

\end{document}